\documentclass{article}

\usepackage{amsmath}
\usepackage{amssymb}
\usepackage{epsfig}
\usepackage{graphics}
\usepackage{graphicx}
\usepackage{color}
\usepackage{float}
\usepackage{tikz}
\usepackage{amsthm}
\newtheorem{theorem}{Theorem}
\newtheorem{proposition}[theorem]{Proposition}
\newtheorem{lemma}[theorem]{Lemma}
\newtheorem{corollary}[theorem]{Corollary}
\newtheorem*{remark}{Remark}
\newtheorem*{theorem1}{Theorem 1.3}
\newtheorem*{theorem2}{Theorem 1.5}
\newtheorem*{corollary3}{Corollary 1.4}
\newtheorem{conjecture}{Conjecture}

\numberwithin{theorem}{section}
\numberwithin{conjecture}{section}
\numberwithin{figure}{section}
\theoremstyle{definition}
\newtheorem{example}{Example}[section]
\makeatletter
\def\blfootnote{\xdef\@thefnmark{}\@footnotetext}
\makeatother
\title{Vexillary Grothendieck Polynomials via Bumpless Pipe Dreams}
\date{}
\author{Elena S. Hafner}

\begin{document}

\thispagestyle{empty}

\maketitle
\begin{abstract}
\noindent Recent work of Pechenik, Speyer, and Weigandt proved a formula for the degree of any Grothendieck polynomial.  A distinct formula for the degree of vexillary Grothendieck polynomials was proven by Rajchgot, Robichaux, and Weigandt.  We give a new proof of Pechenik, Speyer, and Weigandt's formula in the special case of vexillary permutations and characterize the set of bumpless pipe dreams which contribute maximal degree terms to the Grothendieck polynomial in this case.  Furthermore, we use this characterization to draw connections between the Pechenik-Speyer-Weigandt and Rajchgot-Robichaux-Weigandt formulas.  We also use bumpless pipe dreams to prove new results about the support of vexillary Grothendieck polynomials, addressing special cases of conjectures of M\'esz\'aros, Setiabrata, and St.~Dizier. 
\end{abstract}
\section{Introduction}  
Introduced by Lascoux and Sch\"utzenberger \cite{lascoux1982structure}, Grothendieck polynomials are a family of polynomials, indexed by the set of permutations $S_n$, which represent K-theoretic classes of the complete flag variety.  They are defined as follows.\\ \\
Define the \textbf{divided difference operator} $\partial_i$ by $\partial_i(f)=\frac{f-s_if}{x_i-x_{i+1}}$ where $f$ is a polynomial in $x_1,\ldots,x_{n+1}$ and $s_i$ acts on $f$ by interchanging the variables $x_i$ and $x_{i+1}$.  Let $\pi_i(f)=\partial_i(f-x_{i+1}f)$. The \textbf{Grothendieck polynomial} associated to a permutation $\omega \in S_n$ is then given by 
\[\mathfrak{G}_{\omega_0}(x_1,\ldots,x_n)=x_1^{n-1} x_2^{n-2} \cdots x_{n-1}\]
where $\omega_0=n \text{  } n-1 \text{  } n-2 \text{  } \ldots \text{  } 1$ is the permutation in $S_n$ with maximal length, and
\[\mathfrak{G}_{\omega s_i}=\pi_i \mathfrak{G}_{\omega}(x_1,\ldots,x_n)\]
where $s_i=(i, i+1)$ is a simple transposition such that $\omega s_i$ has one fewer inversion than $\omega$. \\ \\
Various combinatorial models for Grothendieck polynomials are described by \cite{FK94, knutson2005grobner, Lenart_groth, weigandt2020bumpless}.  This article focuses on bumpless pipe dreams which were introduced by Lam, Lee, and Shimozono \cite{lam2021stable} to model Schubert polynomials.  Weigandt \cite{weigandt2020bumpless} extended the model to yield a formula for Grothendieck polynomials.  \\ \\
A \textbf{bumpless pipe dream} (BPD) is a tiling of the $n \times n$ grid with the tiles   \\
\begin{center}
    \begin{tikzpicture}
\draw[gray, thin] (0,0) rectangle (.5,.5);
\draw[black, thick] (.25,0) -- (.25,.5);
\draw[black,thick] (0,.25) -- (.5,.25);
\draw[gray, thin] (1,0) rectangle (1.5,.5);
\draw[black, thick] (1.25,0) -- (1.25,.5);
\draw[gray, thin] (2,0) rectangle (2.5,.5);
\draw[black,thick] (2,.25) -- (2.5,.25);
\draw[gray, thin] (3,0) rectangle (3.5,.5);
\draw[black,thick] (3,.25) .. controls (3.25,.25) .. (3.25,.5);
\draw[gray, thin] (4,0) rectangle (4.5,.5);
\draw[black,thick] (4.25,0) .. controls (4.25,.25) .. (4.5,.25);
\draw[gray, thin] (5,0) rectangle (5.5,.5);
\end{tikzpicture}
\end{center} 
such that they form a network of $n$ pipes, each running from the bottom edge of the grid to the right edge \cite{lam2021stable, weigandt2020bumpless}.  To any such tiling $P$, there is an associated permutation $\omega$ given by labeling the pipes $1$ through $n$ along the bottom edge and then reading off the labels on the right edge, ignoring any crossings after the first between each pair of pipes.  In other words, replace any redundant crossings with bump tiles  
\begin{tikzpicture}
\draw[gray, thin] (0,0) rectangle (.5,.5);
\draw[black,thick] (0,.25) .. controls (.25,.25) .. (.25,.5);
\draw[black,thick] (.25,0) .. controls (.25,.25) .. (.5,.25);
\end{tikzpicture} 
 before tracing the path of each pipe to the right edge \textemdash \space see Figure \ref{fig:nonreduced}.  For a fixed permutation $\omega$, we notate the set of all BPDs associated to $\omega$ as $Pipes(\omega)$. \\ 
 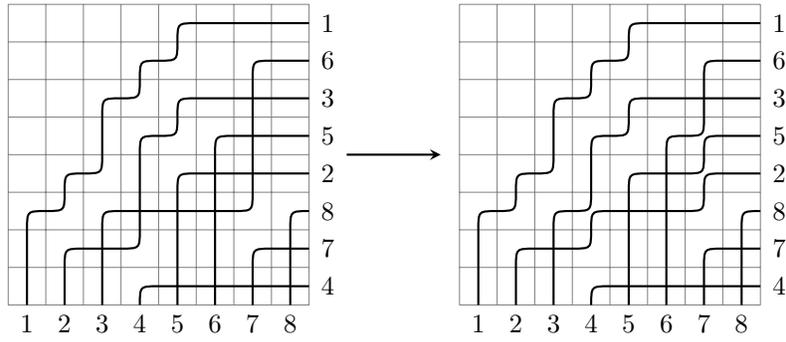
\begin{figure}
    \centering
    \scalebox{1}{
    \begin{tikzpicture}
\draw[step=.5cm,gray,very thin] (0,0) grid (4,4);
\node at (.25,-.25) {1};
\node at (.75,-.25) {2};
\node at (1.25,-.25) {3};
\node at (1.75,-.25) {4};
\node at (2.25,-.25) {5};
\node at (2.75,-.25) {6};
\node at (3.25,-.25) {7};
\node at (3.75,-.25) {8};
\node at (4.25,3.75) {1};
\node at (4.25,3.25) {6};
\node at (4.25,2.75) {3};
\node at (4.25,2.25) {5};
\node at (4.25,1.75) {2};
\node at (4.25,1.25) {8};
\node at (4.25,.75) {7};
\node at (4.25,.25) {4};
\draw[-stealth, black, thick] (4.5,2) -- (5.75,2);
\draw[step=.5cm,gray,very thin] (5.999,0) grid (10,4);
\node at (6.25,-.25) {1};
\node at (6.75,-.25) {2};
\node at (7.25,-.25) {3};
\node at (7.75,-.25) {4};
\node at (8.25,-.25) {5};
\node at (8.75,-.25) {6};
\node at (9.25,-.25) {7};
\node at (9.75,-.25) {8};
\node at (10.25,3.75) {1};
\node at (10.25,3.25) {6};
\node at (10.25,2.75) {3};
\node at (10.25,2.25) {5};
\node at (10.25,1.75) {2};
\node at (10.25,1.25) {8};
\node at (10.25,.75) {7};
\node at (10.25,.25) {4};

\draw[black,thick] (.25,0) -- (.25,1).. controls (.25,1.25) .. (.5,1.25) .. controls (.75,1.25) .. (.75,1.5) .. controls (.75,1.75) .. (1,1.75) .. controls (1.25,1.75) .. (1.25,2) -- (1.25,2.5) .. controls (1.25,2.75) .. (1.5,2.75).. controls (1.75,2.75) .. (1.75,3) .. controls (1.75,3.25) .. (2,3.25) .. controls (2.25,3.25) .. (2.25,3.5) .. controls (2.25,3.75) .. (2.5,3.75) -- (4,3.75) ;

\draw[black,thick] (.75,0) -- (.75,.5) .. controls (.75, .75) .. (1,.75) -- (1.5,.75) .. controls (1.75,.75) .. (1.75,1) -- (1.75,2) .. controls (1.75,2.25) .. (2,2.25) .. controls (2.25,2.25) .. (2.25,2.5) .. controls (2.25,2.75) .. (2.5,2.75) -- (4,2.75) ;

\draw[black,thick] (1.25,0) -- (1.25,1) .. controls (1.25,1.25) .. (1.5,1.25) -- (3,1.25) .. controls (3.25,1.25) .. (3.25,1.5) -- (3.25,3) .. controls (3.25,3.25) .. (3.5,3.25) -- (4,3.25);

\draw[black,thick] (1.75,0) .. controls (1.75,.25) .. (2,.25) -- (4,.25);

\draw[black,thick] (2.25,0) -- (2.25,1.5) .. controls (2.25,1.75) .. (2.5,1.75) -- (4,1.75);

\draw[black,thick] (2.75,0) -- (2.75, 2) .. controls (2.75,2.25) .. (3,2.25) -- (4,2.25);

\draw[black,thick] (3.25,0) -- (3.25,.5) .. controls (3.25,.75) .. (3.5,.75) -- (4,.75);

\draw[black,thick] (3.75,0) -- (3.75,1) .. controls (3.75,1.25) .. (4,1.25);

\draw[black,thick] (6.25,0) -- (6.25,1).. controls (6.25,1.25) .. (6.5,1.25) .. controls (6.75,1.25) .. (6.75,1.5) .. controls (6.75,1.75) .. (7,1.75) .. controls (7.25,1.75) .. (7.25,2) -- (7.25,2.5) .. controls (7.25,2.75) .. (7.5,2.75).. controls (7.75,2.75) .. (7.75,3) .. controls (7.75,3.25) .. (8,3.25) .. controls (8.25,3.25) .. (8.25,3.5) .. controls (8.25,3.75) .. (8.5,3.75) -- (10,3.75) ;

\draw[black,thick] (6.75,0) -- (6.75,.5) .. controls (6.75, .75) .. (7,.75) -- (7.5,.75) .. controls (7.75,.75) .. (7.75,1) .. controls (7.75,1.25) .. (8,1.25) -- (9,1.25) .. controls (9.25,1.25) .. (9.25,1.5) .. controls (9.25,1.75) .. (9.5,1.75) -- (10,1.75);

\draw[black,thick] (7.25,0) -- (7.25,1) .. controls (7.25,1.25) .. (7.5,1.25) .. controls (7.75,1.25) .. (7.75,1.5) -- (7.75,2) .. controls (7.75,2.25) .. (8,2.25) .. controls (8.25,2.25) .. (8.25,2.5) .. controls (8.25,2.75) .. (8.5,2.75) -- (10,2.75);

\draw[black,thick] (7.75,0) .. controls (7.75,.25) .. (8,.25) -- (10,.25);

\draw[black,thick] (8.25,0) -- (8.25,1.5) .. controls (8.25,1.75) .. (8.5,1.75) -- (9,1.75) .. controls (9.25,1.75) .. (9.25,2) .. controls (9.25,2.25) .. (9.5,2.25) -- (10,2.25);

\draw[black,thick] (8.75,0) -- (8.75, 2) .. controls (8.75,2.25) .. (9,2.25) .. controls (9.25,2.25) .. (9.25,2.5) -- (9.25,3) .. controls (9.25,3.25) .. (9.5,3.25) -- (10,3.25);

\draw[black,thick] (9.25,0) -- (9.25,.5) .. controls (9.25,.75) .. (9.5,.75) -- (10,.75);

\draw[black,thick] (9.75,0) -- (9.75,1) .. controls (9.75,1.25) .. (10,1.25);

\end{tikzpicture}}
    \caption{The left diagram shows a non-reduced bumpless pipe dream, and the right diagram shows which crosses are ignored in determining the associated permutation $\omega=16352874$. }
    \label{fig:nonreduced}
\end{figure}
\\
 Define the \textbf{Rothe bumpless pipe dream} $P_R(\omega)$ to be the unique BPD which has \begin{tikzpicture}
\draw[gray, thin] (0,0) rectangle (.5,.5);
\draw[black,thick] (.25,0) .. controls (.25,.25) .. (.5,.25);
\end{tikzpicture} tiles
in the squares $(i,\omega(i))$ for every $i$ and no \begin{tikzpicture}
\draw[gray, thin] (0,0) rectangle (.5,.5);
\draw[black,thick] (0,.25) .. controls (.25,.25) .. (.25,.5);
\end{tikzpicture} tiles.  We will subsequently refer to \begin{tikzpicture}
\draw[gray, thin] (0,0) rectangle (.5,.5);
\draw[black,thick] (.25,0) .. controls (.25,.25) .. (.5,.25);
\end{tikzpicture} tiles as \textbf{down-elbows} and \begin{tikzpicture}
\draw[gray, thin] (0,0) rectangle (.5,.5);
\draw[black,thick] (0,.25) .. controls (.25,.25) .. (.25,.5);
\end{tikzpicture} tiles as \textbf{up-elbows}. \\ \\
Let $U(P)$ be the set of all up-elbow tiles in $P$, and let $D(P)$ be the set of all blank tiles.  A \textbf{marked bumpless pipe dream} \cite{weigandt2020bumpless} is an ordered pair $(P,S)$ where $P$ is a bumpless pipe dream and $S$ is some subset of the set $U(P)$ of up-elbow tiles in $P$.  The set of all marked BPDs for $\omega$ is denoted $MPipes(\omega )$.  The Grothendieck polynomial associated to $\omega$ is then given by \cite{weigandt2020bumpless}\\
\[\mathfrak{G}_{\omega}(x_1,\ldots,x_n)=\sum_{(P,S) \in MPipes( \omega)}(-1)^{|D(P)| + |S| - \ell(\omega)}\left( \prod_{(i,j) \in D(P) \cup S}x_i \right)\]
where $D(P)$ is the set of all blank tiles in $P$.\\ \\
An explicit formula for the degree of a Grothendieck polynomial was proven in a recent paper by Pechenik, Speyer, and Weigandt \cite{Speyerslides}.  The \textbf{Rajchgot code} of a permutation $\omega \in S_n$ is defined to be $(r_1, r_2, \ldots, r_n)$ where $r_i$ is the minimum number of elements which must be removed from $\omega(i), \omega(i+1),\ldots, \omega(n)$ to form an increasing sequence beginning with $\omega(i)$. \\
\begin{theorem}[\cite{Speyerslides}, Theorem 1.1]
\label{PSWdegree}
The degree of the Grothendieck polynomial $\mathfrak{G}_{\omega}(x_1,\ldots,x_n)$ is given by $\sum_{i=1}^nr_i$.
\end{theorem}
\noindent The \textbf{Lehmer code} is defined to be $(c_1, \ldots, c_n)$ where $c_i=|\{j : i<j, \omega(i)>\omega(j)\}|$.  Notice that $r_i \geq c_i$.  Recall that the \textbf{length} of $\omega$ is $\ell(\omega)=\sum_{i=1}^n c_i$ and, analogously to the above formula, equals the degree of the Schubert polynomial $\mathfrak{S}_{\omega}(x_1,\ldots,x_n)$ \cite{lascoux1982polynomes}.  \\ \\
\noindent Other combinatorial degree formulas have been given for special cases of Grothendieck polynomials \cite{symmetricdegree, rajchgot2022castelnuovomumford}.  In particular, recall that a \textbf{vexillary permutation} is one that is 2143-avoiding.  The following formula for the degree of Grothendieck polynomials for vexillary permutations was introduced by Rajchgot, Robichaux, and Weigandt \cite{rajchgot2022castelnuovomumford}.  \\ 
\\
\noindent Let $\omega \in S_n$ be vexillary.  To every blank tile $(i,j)$ in $P_R(\omega)$, assign the label 
\[r(i,j)=|\{(k,\omega(k))|k<i,\omega(k)<j\}|.\]  
This can be thought of as the number of pipes which pass above and to the left of $(i,j)$ in $P_R(\omega)$.  Let $\lambda(\omega)$ be the Young diagram which has the same number of tiles in each diagonal as in the corresponding diagonal of $D(P_R(\omega))$, and fill each tile with $r(i,j)$ for the corresponding $(i,j) \in D(P_R(\omega))$.  Write $\tau_k(\omega)$ for the portion of $\lambda(\omega)$ filled with values greater than or equal to $k$ (see Figure \ref{fig:RRWLabeling}).  
\begin{figure}
    \centering
\begin{tikzpicture}
 \draw[step=.5cm,gray,very thin] (0,0) grid (4,4);
    \node at (.25,-.25) {1};
    \node at (.75,-.25) {2};
    \node at (1.25,-.25) {3};
    \node at (1.75,-.25) {4};
    \node at (2.25,-.25) {5};
    \node at (2.75,-.25) {6};
    \node at (3.25,-.25) {7};
    \node at (3.75,-.25) {8};
    \node at (4.25,3.75) {1};
    \node at (4.25,3.25) {8};
    \node at (4.25,2.75) {2};
    \node at (4.25,2.25) {7};
    \node at (4.25,1.75) {3};
    \node at (4.25,1.25) {5};
    \node at (4.25,.75) {6};
    \node at (4.25,.25) {4};
    \draw[black,thick] (.25,0) -- (.25,3.5) .. controls (.25,3.75) .. (.5,3.75) -- (4,3.75);
    \draw[black,thick] (.75,0) -- (.75, 2.5) .. controls (.75,2.75) .. (1,2.75) -- (4,2.75);
    \draw[black,thick] (1.25,0) -- (1.25,1.5) .. controls (1.25,1.75) .. (1.5,1.75) -- (4,1.75);
    \draw[black,thick] (1.75,0) .. controls (1.75,.25) .. (2,.25) -- (4,.25);
    \draw[black,thick] (2.25,0) -- (2.25,1) .. controls (2.25,1.25) .. (2.5,1.25) -- (4,1.25);
    \draw[black,thick] (2.75,0) -- (2.75,.5) .. controls (2.75,.75) .. (3,.75) -- (4,.75);
    \draw[black,thick] (3.25,0) -- (3.25,2) .. controls (3.25,2.25) .. (3.5,2.25) -- (4,2.25);
    \draw[black,thick] (3.75,0) -- (3.75,3) .. controls (3.75,3.25) .. (4,3.25);
    \begin{scope}[shift={(5,0)}]
    \draw[step=.5cm,gray,very thin] (0,3.499) grid (3,4);
    \draw[step=.5cm,gray,very thin] (0,2.99) grid (2,3.5);
    \draw[step=.5cm,gray,very thin] (0,1.99) grid (.5,3);
    \node at (.25,3.75) {1};
    \node at (.75,3.75) {1};
    \node at (1.25,3.75) {1};
    \node at (1.75,3.75) {1};
    \node at (2.25,3.75) {1};
    \node at (2.75,3.75) {1};
    \node at (.25,3.25) {2};
    \node at (.75,3.25) {2};
    \node at (1.25,3.25) {2};
    \node at (1.75,3.25) {2};
    \node at (.25,2.75) {3};
    \node at (.25,2.25) {3};
    \end{scope}
      \node at (3.25,3.25) {1};
    \node at (.75,3.25) {1};
    \node at (1.25,3.25) {1};
    \node at (1.75,3.25) {1};
    \node at (2.25,3.25) {1};
    \node at (2.75,3.25) {1};
    \node at (1.25,2.25) {2};
    \node at (1.75,2.25) {2};
    \node at (2.25,2.25) {2};
    \node at (2.75,2.25) {2};
    \node at (1.75,1.25) {3};
    \node at (1.75,.75) {3};
\end{tikzpicture} 
    \caption{Shows the labeling $r(i,j)$ on $P_R(\omega)$ and the corresponding filling of $\lambda(\omega)$}
    \label{fig:RRWLabeling}
\end{figure}
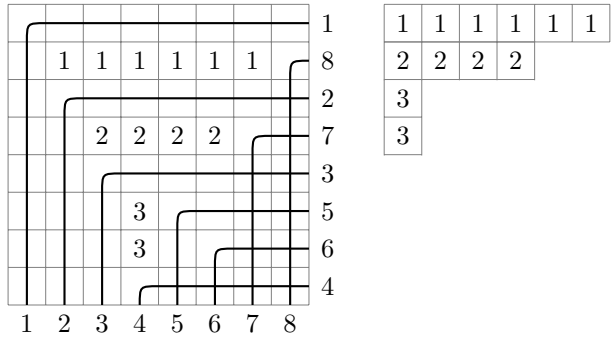 \\ \\
\noindent An \textbf{antidiagonal path} is a subset $\{(i_1,j_1) \ldots (i_k,j_k)\} \subset \left[n\right] \times \left[n\right]$ satisfying $i_1 > i_2 > \ldots > i_k$ and $j_1 <j_2 < \ldots < j_k$.  For any subset $S \subset \left[n\right] \times \left[n\right]$, let $\rho_a(S)$ denote the length of the longest antidiagonal path contained in $S$.  
\begin{theorem}[\cite{rajchgot2022castelnuovomumford}, Theorem 2.4]
\label{RRWdegree}
If $\omega \in S_n$ is vexillary, then \\
$deg(\mathfrak{G}_{\omega})=\ell(\omega)+\sum_{k=1}^{n}\rho_a(\tau_k(\omega))$
\end{theorem}
\noindent In this article, we prove the following characterization of the marked BPDs which contribute maximal degree terms to $G_{\omega}$ in the case where $\omega$ is vexillary. 
\begin{theorem}
\label{main} Suppose that $P$ is a bumpless pipe dream corresponding to a vexillary permutation $\omega$.  Then the $i^{th}$ pipe has at most $r_{\omega^{-1}(i)}-c_{\omega^{-1}(i)}$ up-elbows.  In particular, if the marked bumpless pipe dream $(P,U(P))$ corresponds to a maximal degree term of $\mathfrak{G}_{\omega}$, the $i^{th}$ pipe will have exactly $r_{\omega^{-1}(i)}-c_{\omega^{-1}(i)}$ up-elbows.  
\end{theorem}
\noindent Note that this result serves as an alternate proof of Theorem \ref{PSWdegree} in the case of vexillary permutations.  The proof of Theorem \ref{main} also allows us to prove the following connection between the formulas in Theorems \ref{PSWdegree} and \ref{RRWdegree}. \\ \\
For a fixed vexillary permutation $\omega$, let $rank_{\omega}(i)$ be the length of the longest increasing sequence of $j \in \left[n \right]$ such that $j \leq i \text{ and } \omega(j)\leq\omega(i)$. 
\begin{corollary}
Let $\omega \in S_n$ be vexillary.  Then $\sum_{\{i : rank_{\omega}(i)=k\}}(r_{\omega^{-1}(i)}-c_{\omega^{-1}(i)})=\rho_a(\tau_k(\omega))$.
\end{corollary}
\noindent Additionally, Theorem \ref{main} allows us to prove new results about the support of Grothendieck polynomials.  In particular, M\'esz\'aros, Setiabrata, and St.~Dizier \cite{conjectures} conjectured the following.
\begin{conjecture}[\cite{conjectures}, Conjecture 1.2]
\label{conjsides}
Let $\omega \in S_n$, and let $x_1^{i_1}x_2^{i_2}\ldots x_n^{i_n}$ be any monomial with a nonzero coefficient and non-maximal degree in $\mathfrak{G}_{\omega}(x_1,\ldots,x_n)$.  There exists a monomial $x_1^{j_1}x_2^{j_2}\ldots x_n^{j_n}$ with nonzero coefficient in $\mathfrak{G}_{\omega}(x_1,\ldots,x_n)$ such that $j_k=i_k+1$ for some $1 \leq k \leq n$ and $j_l=i_l$ for all other indices $l$.
\end{conjecture}
\begin{conjecture}[\cite{conjectures}, Conjecture 1.3]
\label{conjbetween}
Let $\omega \in S_n$, and suppose that $p_1(x_1, \ldots, x_n)$ and $p_2(x_1, \ldots, x_n)$ are monomials with nonzero coeffiecient in $\mathfrak{G}_{\omega}$ such that $p_1 | p_2$.  Then any monomial $q(x_1, \ldots, x_n)$ satisfying $p_1 | q $ and $q | p_2$ must also have nonzero coefficient in $\mathfrak{G}_{\omega}$.
\end{conjecture}
\noindent We give proofs of these conjectures in the special case where $\omega$ is vexillary (in Proposition \ref{upByOne} and Theorem \ref{VexBetween}, respectively). \\ \\
Furthermore, we address the following question posed by Weigandt \cite{WeigandtTalk}:
\begin{quote}
``Can you find \ldots permutation codes that tell you the lex$\left[\text{icographically}\right]$ last monomials in each degree of a Grothendieck polynomial?"    
\end{quote}
Recall that the lexicographically last monomial in each homogeneous component $\mathfrak{G}_{\omega}^{(k)}(\mathbf{x})$ is the leading term in any monomial order satisfying $x_1<x_2< \cdots <x_n$.  When $\omega$ is vexillary, we prove the following formula for each such leading term.
\begin{theorem}
\label{leadingterms} 
Suppose that $\omega$ is vexillary, and fix $k$ such that $deg(\mathfrak{S}_{\omega}(x)) \leq k \leq deg(\mathfrak{G}_{\omega}(x))$.  Let $j \in \{1,\ldots,n\}$ be the smallest possible index such that $\sum_{i=j}^n(r_i-c_i) \leq k-deg(\mathfrak{S}_{\omega}(x))$.  If $e_i=r_i$ for all $i \geq j$, $e_{j-1}=c_{j-1}+k-deg(\mathfrak{S}_{\omega}(x))-\sum_{i=j}^n(r_i-c_i)$, and $e_i=c_i$ for all $i<j-1$, then the leading term of $\mathfrak{G}_{\omega}^{(k)}(\mathbf{x})$ in any term order with $x_1<x_2< \cdots <x_n$ is a scalar multiple of $x_1^{e_1} \cdots x_n^{e_n}$.
\end{theorem} 
\noindent In section 2, we give some existing results about bumpless pipe dreams as well as the proof of Theorem \ref{main}.  We also prove some consequences and conjecture a possible generalization of that result.  Section 3 includes facts about the support of vexillary Grothendieck polynomials including Theorem \ref{leadingterms} and the vexillary cases of Conjectures \ref{conjsides} and \ref{conjbetween}.  \\

\section{Vexillary Bumpless Pipe Dreams}  
We begin by recalling a method for generating BPDs.  A \textbf{droop} \cite{lam2021stable} is a move of the form \\
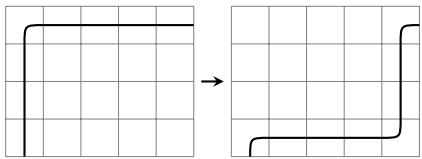
\begin{center}
    \scalebox{1}{\begin{tikzpicture}
\draw[step=.5cm,gray,very thin] (0,0) grid (2.5,2);
\draw[step=.5cm,gray,very thin] (2.999,0) grid (5.5,2);
\draw[-stealth, black, thick] (2.6,1) -- (2.9,1);
\draw[black, thick] (.25,0) -- (.25,1.5);
\draw[black, thick] (.25,1.5) .. controls (.25,1.75) .. (.5,1.75);
\draw[black,thick] (.5,1.75) -- (2.5,1.75);
\draw[black,thick] (3.25,0) .. controls (3.25,.25) .. (3.5,.25);
\draw[black,thick] (3.5,.25) -- (5,.25);
\draw[black,thick] (5, .25) .. controls (5.25,.25) .. (5.25,.5);
\draw[black,thick] (5.25,.5)--(5.25,1.5);
\draw[black,thick] (5.25,1.5).. controls (5.25,1.75) .. (5.5,1.75);
\end{tikzpicture}}
\end{center}
where the region shown contains no other (unpictured) elbow tiles.  Similarly, \textbf{K-theoretic droops} are moves of one of the following forms.  \\
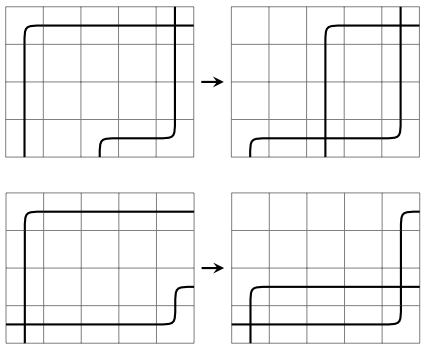
\begin{center}
   \scalebox{1}{\begin{tikzpicture}
\draw[step=.5cm,gray,very thin] (0,0) grid (2.5,2);
\draw[step=.5cm,gray,very thin] (2.999,0) grid (5.5,2);
\draw[-stealth, black, thick] (2.6,1) -- (2.9,1);
\draw[black, thick] (.25,0) -- (.25,1.5);
\draw[black, thick] (.25,1.5) .. controls (.25,1.75) .. (.5,1.75);
\draw[black,thick] (.5,1.75) -- (2.5,1.75);
\draw[black,thick] (1.25,0) .. controls (1.25,.25) .. (1.5,.25) -- (2,.25) .. controls (2.25,.25) .. (2.25,.5) -- (2.25,2);

\draw[black,thick] (3.25,0) .. controls (3.25,.25) .. (3.5,.25);
\draw[black,thick] (3.5,.25) -- (5,.25);
\draw[black,thick] (5, .25) .. controls (5.25,.25) .. (5.25,.5);
\draw[black,thick] (5.25,.5)--(5.25,2);
\draw[black,thick] (4.25,0) -- (4.25,1.5) .. controls (4.25,1.75) .. (4.5,1.75) -- (5.5,1.75);
\end{tikzpicture}}

    \scalebox{1}{\begin{tikzpicture}
\draw[step=.5cm,gray,very thin] (0,0) grid (2.5,2);
\draw[step=.5cm,gray,very thin] (2.999,0) grid (5.5,2);
\draw[-stealth, black, thick] (2.6,1) -- (2.9,1);
\draw[black, thick] (.25,0) -- (.25,1.5);
\draw[black, thick] (.25,1.5) .. controls (.25,1.75) .. (.5,1.75);
\draw[black,thick] (.5,1.75) -- (2.5,1.75);
\draw[black,thick] (0,.25) -- (2,.25) .. controls (2.25,.25) .. (2.25,.5) .. controls (2.25,.75) .. (2.5,.75);

\draw[black,thick] (3,.25) -- (5,.25);
\draw[black,thick] (5, .25) .. controls (5.25,.25) .. (5.25,.5);
\draw[black,thick] (5.25,.5)--(5.25,1.5);
\draw[black,thick] (5.25,1.5).. controls (5.25,1.75) .. (5.5,1.75);
\draw[black,thick] (3.25,0) -- (3.25,.5) .. controls (3.25,.75) .. (3.5,.75) -- (5.5,.75);
\end{tikzpicture}}
\end{center} 
\noindent A bumpless pipe dream is called \textbf{reduced} if each pair of pipes crosses at most once, and the set of reduced BPDs is denoted $RPipes(\omega)$.  It is known that every $P \in RPipes(\omega)$ can be reached from $P_R(\omega)$ by a series of droops \cite{lam2021stable}.  Furthermore, every $P \in Pipes(\omega)$ can be reached by a series of droops and K-theoretic droops \cite{weigandt2020bumpless}.\\ \\
This model can be simplified in the case where $\omega$ is vexillary. \\
\begin{lemma}[\cite{weigandt2020bumpless}, Lemma 7.2]
A permutation is vexillary if and only if the permutation has no non-reduced bumpless pipe dreams.
\end{lemma} 
\noindent In particular, we have the following specialization of droop moves. \\
\begin{lemma}[\cite{weigandt2020bumpless}, Lemma 7.4]
\label{localmoves}
For a vexillary permutation $\omega$, $Pipes(\omega)$ is connected by the following set of local moves.  \\
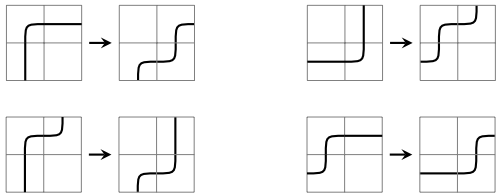
\begin{center}
\scalebox{1}{\begin{tikzpicture}
\draw[step=.5cm,gray,very thin] (0,0) grid (1,1);
\draw[black, thick] (.25,.5) .. controls (.25,.75) .. (.5,.75);
\draw[black, thick] (.5,.75) -- (1,.75);
\draw[black, thick] (.25,0) -- (.25,.5);
\draw[black, thick] (1.75,0) .. controls (1.75,.25) .. (2,.25);
\draw[black, thick] (2,.25) .. controls (2.25,.25) .. (2.25,.5);
\draw[black, thick] (2.25,.5) .. controls (2.25,.75) .. (2.5,.75);
\draw[step=.5cm,gray,very thin] (1.4999,0) grid (2.5,1);
\draw[-stealth, black, thick] (1.1,.5) -- (1.4,.5);

\draw[step=.5cm,gray,very thin] (3.999,0) grid (5,1);
\draw[black, thick] (4,.25) -- (4.5,.25) .. controls (4.75,.25) .. (4.75,.5) -- (4.75,1);
\draw[black, thick] (5.5,.25) .. controls (5.75,.25) .. (5.75,.5) .. controls (5.75,.75) .. (6,.75) .. controls (6.25,.75) .. (6.25,1);
\draw[step=.5cm,gray,very thin] (5.4999,0) grid (6.5,1);
\draw[-stealth, black, thick] (5.1,.5) -- (5.4,.5);

\end{tikzpicture}}

\scalebox{1}{\begin{tikzpicture}
\draw[step=.5cm,gray,very thin] (0,0) grid (1,1);
\draw[black, thick] (.25,.5) .. controls (.25,.75) .. (.5,.75);
\draw[black, thick] (.5,.75) .. controls (.75,.75) .. (.75,1);
\draw[black, thick] (.25,0) -- (.25,.5);
\draw[black, thick] (1.75,0) .. controls (1.75,.25) .. (2,.25);
\draw[black, thick] (2,.25) .. controls (2.25,.25) .. (2.25,.5);
\draw[black, thick] (2.25,.5) -- (2.25,1);
\draw[step=.5cm,gray,very thin] (1.4999,0) grid (2.5,1);
\draw[-stealth, black, thick] (1.1,.5) -- (1.4,.5);

\draw[step=.5cm,gray,very thin] (3.999,0) grid (5,1);
\draw[black, thick] (4.25,.5) .. controls (4.25,.75) .. (4.5,.75);
\draw[black, thick] (4,.25) .. controls (4.25,.25) .. (4.25,.5);
\draw[black, thick] (4.5,.75) -- (5,.75);
\draw[black, thick] (6.25,.5) .. controls (6.25,.75) .. (6.5,.75);
\draw[black, thick] (6,.25) .. controls (6.25,.25) .. (6.25,.5);
\draw[black, thick] (5.5,.25) -- (6,.25);
\draw[step=.5cm,gray,very thin] (5.4999,0) grid (6.5,1);
\draw[-stealth, black, thick] (5.1,.5) -- (5.4,.5);

\end{tikzpicture}}
\end{center}
\end{lemma}
\begin{lemma}
\label{VexPipesOrdering}
Let $\omega$ be a vexillary permutation on $n$ elements, and let $P \in Pipes(\omega)$.  
\begin{enumerate}
    \item [(i)] For any $i,j$ such that $\omega^{-1}(i) < \omega^{-1}(j)$ or $i < j$, the $i^{th}$ pipe will never contain any elbow tiles which are southeast of the $j^{th}$ pipe.  
    \item[(ii)] $P$ can be constructed from $P_R(\omega)$ by using droop moves to position pipe $\omega(n)$, then pipe $\omega(n-1)$, and so on through pipe $\omega(1)$.  Alternatively, we can construct $P$ by first positioning pipe $n$, then pipe $n-1$, and so on through pipe $1$.
\end{enumerate}
\end{lemma}
\begin{proof}
Follows immediately from the fact that $Pipes(\omega)$ is connected by the local moves of Lemma \ref{localmoves}.  
\end{proof}
\noindent We can now prove Theorem \ref{main}, beginning with the following lemma. \\
\begin{lemma}
\label{construction}
For any vexillary $\omega \in S_n$, there exists a BPD such that each pipe $i$ has at least $r_{\omega^{-1}(i)}-c_{\omega^{-1}(i)}$ up-elbows.
\end{lemma}
\begin{proof}
Consider the Rothe BPD $P_R(\omega)$.  For each pipe $i$, we can define a diagram $D_i$ as follows.  Begin with an $m \times m$ grid where $m=n-\omega^{-1}(i)-c_{\omega^{-1}(i)}$, and shade the squares corresponding to the set of empty squares underneath pipe $i$ in $P_R(\omega)$, pushed as far to the top-left corner as possible.  Fill each remaining square above the main anti-diagonal with an X, and leave any square on or below the main anti-diagonal blank (see Figure \ref{fig:Top-Deg Construction}).
\begin{figure}
    \centering
    \scalebox{1}{
    \begin{tikzpicture}
    \draw[step=.5cm,gray,very thin] (0,0) grid (4,4);
    \node at (.25,-.25) {1};
    \node at (.75,-.25) {2};
    \node at (1.25,-.25) {3};
    \node at (1.75,-.25) {4};
    \node at (2.25,-.25) {5};
    \node at (2.75,-.25) {6};
    \node at (3.25,-.25) {7};
    \node at (3.75,-.25) {8};
    \node at (4.25,3.75) {1};
    \node at (4.25,3.25) {8};
    \node at (4.25,2.75) {2};
    \node at (4.25,2.25) {7};
    \node at (4.25,1.75) {3};
    \node at (4.25,1.25) {5};
    \node at (4.25,.75) {6};
    \node at (4.25,.25) {4};
    \draw[black,thick] (.25,0) -- (.25,3.5) .. controls (.25,3.75) .. (.5,3.75) -- (4,3.75);
    \draw[black,thick] (.75,0) -- (.75, 2.5) .. controls (.75,2.75) .. (1,2.75) -- (4,2.75);
    \draw[black,thick] (1.25,0) -- (1.25,1.5) .. controls (1.25,1.75) .. (1.5,1.75) -- (4,1.75);
    \draw[black,thick] (1.75,0) .. controls (1.75,.25) .. (2,.25) -- (4,.25);
    \draw[black,thick] (2.25,0) -- (2.25,1) .. controls (2.25,1.25) .. (2.5,1.25) -- (4,1.25);
    \draw[black,thick] (2.75,0) -- (2.75,.5) .. controls (2.75,.75) .. (3,.75) -- (4,.75);
    \draw[black,thick] (3.25,0) -- (3.25,2) .. controls (3.25,2.25) .. (3.5,2.25) -- (4,2.25);
    \draw[black,thick] (3.75,0) -- (3.75,3) .. controls (3.75,3.25) .. (4,3.25);
    \node at (5.5,3.75) {$D_1=$};
    \filldraw[draw=darkgray, color=lightgray] (6,3.75) rectangle (7.5,4);
    \filldraw[draw=darkgray, color=lightgray] (6,3.5) rectangle (7,3.75);
    \filldraw[draw=darkgray, color=lightgray] (6,3.25) rectangle (6.25,3.5);
    \filldraw[draw=darkgray, color=lightgray] (6,3) rectangle (6.25,3.25);
    \node[scale=.6] at (6.125, 2.625) {X};
    \node[scale=.6] at (6.125, 2.875) {X};
    \node[scale=.6] at (6.375, 2.875) {X};
    \node[scale=.6] at (6.375, 3.125) {X};
    \node[scale=.6] at (6.375, 3.375) {X};
    \node[scale=.6] at (6.625, 3.125) {X};
    \node[scale=.6] at (6.625, 3.375) {X};
    \node[scale=.6] at (6.875, 3.375) {X};
    \node[scale=.6] at (7.125, 3.625) {X};
    \draw[step=.25cm,darkgray,very thin] (5.99,2.249) grid (7.75,4);
    \node at (8.75,3.75) {$D_2=$};
    \filldraw[draw=darkgray, color=lightgray] (9.25,3.75) rectangle (10.25,4);
    \filldraw[draw=darkgray, color=lightgray] (9.25,3.5) rectangle (9.5,3.75);
    \filldraw[draw=darkgray, color=lightgray] (9.25,3.25) rectangle (9.5,3.5);
    \node[scale=.6] at (9.375, 3.125) {X};
    \node[scale=.6] at (9.625, 3.375) {X};
    \node[scale=.6] at (9.625, 3.625) {X};
    \node[scale=.6] at (9.875, 3.625) {X};
    \draw[step=.25cm,darkgray,very thin] (9.249,2.749) grid (10.5,4);
    \node at (5.5,1.75) {$D_3=$};
    \filldraw[draw=darkgray, color=lightgray] (6,1.5) rectangle (6.25,2);
    \node[scale=.6] at (6.375, 1.875) {X};
    \draw[step=.25cm,darkgray,very thin] (5.99,1.249) grid (6.75,2);
    \end{tikzpicture}}
    \caption{Sample construction of the diagrams $D_i$ in the proof of Lemma \ref{construction}.}
    \label{fig:Top-Deg Construction}
\end{figure}
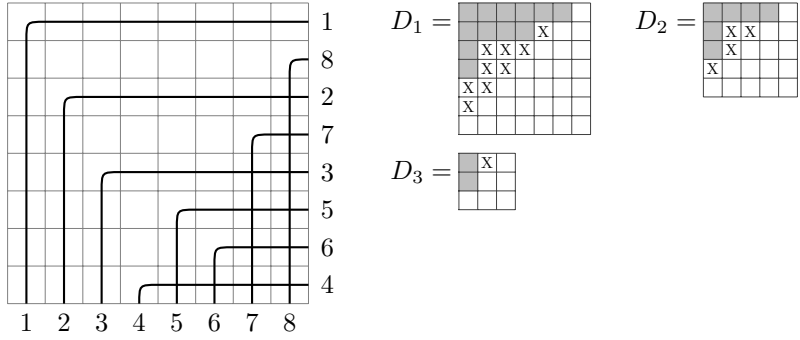 \\ \\
\noindent We will show that whenever there is a column in $D_i$ containing $k$ Xs, there must be an increasing sequence in the one-line notation of $\omega$ beginning with $i$ and having length at least $k+2$.  Suppose that column $j$ contains $k$ Xs (and by definition, $j$ blank squares).  This means that there are $k+j$ rows $l > \omega^{-1}(i)$ of $P_R(\omega)$ which contain fewer than $j$ blank squares underneath pipe $i$ but which satisfy $\omega(l) > i$.  Consider the sequence $\omega^{-1}(i_1), \omega^{-1}(i_2), \ldots, \omega^{-1}(i_{k+j})$ of all such rows, ordered from top to bottom.  Let $\omega^{-1}(i_{a})$ be the first of these rows such that $i_{a} < i_{a-1}$, and let $r$ be the number of indices $l$ in between $i$ and $i_{a-1}$ such that $\omega^{-1}(l)>\omega^{-1}(i_{a-1})$.  Note that there can be at most $j-1$ such indices since each would correspond to a blank square in position $(l,\omega^{-1}(i_{a-1}))$.  \\ \\
Now, let $i_1', \ldots, i_b'$ be the sub-sequence of $i_{a},...,i_{k+j}$ consisting of the indices $i_l$ that satisfy $i_l > i_{a-1}$, and note that since $|\{i_{a},...,i_{k+j}\}-\{i_1', \ldots, i_b'\}|$ must be smaller than $j$, we know that $b \geq k-a+2$.  Furthermore, we know that $i_1'< i_2'< \ldots < i_b'$ (otherwise, we would contradict the assumption that $\omega$ is vexillary), so $i,i_1, \ldots, i_{a-1},i_1', \ldots i_b'$ forms an increasing sequence in $\omega$ of length $k+2$.  \\ \\
For any pipe $i$, the length of the longest anti-diagonal in the shaded portion of $D_i$ is given by $m-1-k_{max}$ where $k_{max}$ is the maximal number of Xs in any column of $D_i$.  By the above argument, we conclude that this anti-diagonal has at least length $r_{\omega^{-1}(i)}-c_{\omega^{-1}(i)}$.  \\ \\
We can now construct the desired BPD.  First, consider the largest $i$ such that $D_i$ has any shaded squares, i.e. the largest $i$ such that $r_{\omega^{-1}(i)}-c_{\omega^{-1}(i)}$ is nonzero.  In $P_R(\omega)$, the set of empty squares underneath pipe $i$ will be a Young diagram matching the shaded portion of $D_i$, and we can perform droop moves so that pipe $i$ has up-elbows in each of the (at least) $r_{\omega^{-1}(i)}-c_{\omega^{-1}(i)}$ squares of the longest anti-diagonal.  Note that this moves all the blank squares that were on or above the longest anti-diagonal above pipe $i$ such that the set of blank tiles in between pipe $i$ and the closest pipe northwest of it form a new Young diagram.  In particular, if we consider a pipe $i'$ such that for every $i>i'$, the number of up-elbows in the $i^{th}$ pipe is the length of the longest diagonal in $D_i$, then the set of blank squares that are below pipe $i'$ but above every lower pipe will form a Young diagram with an anti-diagonal as long as the longest anti-diagonal of $D_{i'}$.  We can then droop pipe $i'$ so that it has up-elbows in each square of this anti-diagonal. \\ \\
Repeating this process for every pipe, thus, generates a BPD in which every pipe $i$ has at least $r_{\omega^{-1}(i)}-c_{\omega^{-1}(i)}$ up-elbows.
\end{proof}
\begin{theorem1}
Suppose that $P$ is a bumpless pipe dream corresponding to a vexillary permutation $\omega$.  Then the $i^{th}$ pipe has at most $r_{\omega^{-1}(i)}-c_{\omega^{-1}(i)}$ up-elbows.  In particular, if the marked bumpless pipe dream $(P,U(P))$ corresponds to a maximal degree term of $\mathfrak{G}_{\omega}$, the $i^{th}$ pipe will have exactly $r_{\omega^{-1}(i)}-c_{\omega^{-1}(i)}$ up-elbows.  
\end{theorem1}
\begin{proof}
Beginning with pipe $\omega(n)$ of $P$ and working upwards, we will define an top label and a side label for every down-elbow in the following way (see Figure \ref{fig:thm1labeling}).  On each pipe $i$, the first down elbow (starting from the bottom) will be given side label $i$, and the last down-elbow will be given top label $i$.  Now, for every down-elbow which does not yet have a top label, pipe $i$ will contain an up-elbow to the right of it.  This unlabeled down-elbow will take the same top label as the closest down-elbow tile beneath this up-elbow.  Similarly, every down-elbow which does not yet have a side label has an up-elbow below it, and this down-elbow will take the same side label as the closest down-elbow tile to the right of that up-elbow.
\begin{figure}
    \centering
    \begin{tikzpicture}
\draw[step=.5cm,gray,very thin] (0,0) grid (4,4);
\node at (.25,-.25) {1};
\node at (.75,-.25) {2};
\node at (1.25,-.25) {3};
\node at (1.75,-.25) {4};
\node at (2.25,-.25) {5};
\node at (2.75,-.25) {6};
\node at (3.25,-.25) {7};
\node at (3.75,-.25) {8};
\node at (4.25,3.75) {1};
\node at (4.25,3.25) {3};
\node at (4.25,2.75) {4};
\node at (4.25,2.25) {7};
\node at (4.25,1.75) {5};
\node at (4.25,1.25) {8};
\node at (4.25,.75) {2};
\node at (4.25,.25) {6};
\draw[black,thick] (.25,0) -- (.25,1) .. controls (.25,1.25) .. (.5,1.25) .. controls (.75,1.25) .. (.75,1.5) --(.75,2.5) .. controls (.75,2.75) .. (1,2.75) .. controls (1.25,2.75) .. (1.25,3) .. controls (1.25,3.25) .. (1.5,3.25) .. controls (1.75,3.25) .. (1.75,3.5) .. controls (1.75,3.75) .. (2,3.75) -- (4,3.75) ;
\draw[black,thick] (.75,0) -- (.75,.5) .. controls (.75,.75) .. (1,.75) -- (4,.75);
\draw[black,thick] (1.25,0) -- (1.25,2) .. controls (1.25,2.25) .. (1.5,2.25) .. controls (1.75,2.25) .. (1.75,2.5) .. controls (1.75,2.75) .. (2,2.75) .. controls (2.25,2.75) .. (2.25,3) .. controls (2.25,3.25) .. (2.5,3.25) -- (4,3.25);
\draw[black,thick] (1.75,0) -- (1.75,1.5) .. controls (1.75,1.75) .. (2,1.75) .. controls (2.25,1.75) .. (2.25,2) .. controls (2.25,2.25) .. (2.5,2.25) .. controls (2.75,2.25) .. (2.75,2.5) .. controls (2.75,2.75) .. (3,2.75) -- (4,2.75);
\draw[black,thick] (2.25,0) -- (2.25,1) .. controls (2.25,1.25) .. (2.5,1.25) .. controls (2.75,1.25) .. (2.75,1.5) .. controls (2.75,1.75) .. (3,1.75) -- (4,1.75);
\draw[black,thick] (2.75,0) .. controls (2.75,.25) .. (3,.25) -- (4,.25);
\draw[black,thick] (3.25,0) -- (3.25,2) .. controls (3.25,2.25) .. (3.5,2.25) -- (4,2.25);
\draw[black,thick] (3.75,0) -- (3.75, 1) .. controls (3.75,1.25) .. (4,1.25);
\node[scale=.6] at (2.6, .125) {6};
\node[scale=.6] at (2.87, .375) {6};
\node[scale=.6] at (.6, .625) {2};
\node[scale=.6] at (.87, .875) {2};
\node[scale=.6] at (3.6, 1.125) {8};
\node[scale=.6] at (3.87, 1.375) {8};
\node[scale=.7] at (2.125, 1.125) {5};
\node[scale=.6] at (2.375, 1.375) {6};
\node[scale=.6] at (2.6, 1.625) {8};
\node[scale=.7] at (2.87, 1.875) {5};
\node[scale=.6] at (3.125, 2.125) {7};
\node[scale=.6] at (3.375, 2.375) {7};
\node[scale=.6] at (1.6, 1.625) {4};
\node[scale=.6] at (1.87, 1.875) {6};
\node[scale=.6] at (2.125, 2.125) {8};
\node[scale=.7] at (2.375, 2.375) {5};
\node[scale=.6] at (2.6, 2.625) {7};
\node[scale=.6] at (2.87, 2.875) {4};
\node[scale=.6] at (1.125, 2.125) {3};
\node[scale=.6] at (1.375, 2.375) {6};
\node[scale=.6] at (1.6, 2.675) {8};
\node[scale=.7] at (1.87, 2.875) {5};
\node[scale=.6] at (2.125, 3.125) {7};
\node[scale=.6] at (2.325, 3.34) {3};
\node[scale=.6] at (.124, 1.125) {1};
\node[scale=.6] at (.375, 1.375) {2};
\node[scale=.7] at (.6, 2.675) {5};
\node[scale=.6] at (.87, 2.875) {6};
\node[scale=.6] at (1.125, 3.125) {8};
\node[scale=.7] at (1.375, 3.375) {5};
\node[scale=.6] at (1.6, 3.625) {7};
\node[scale=.6] at (1.87, 3.875) {1};
\end{tikzpicture}
    \caption{Labeling of the down-elbows in proof of Theorem \ref{main}}
    \label{fig:thm1labeling}
\end{figure}
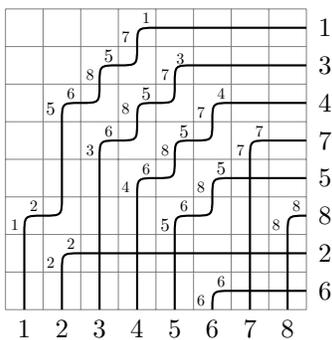 \\ \\
Associate to each up-elbow tile the ordered pair $(a,b)$ where $a$ is the side label of the down-elbow immediately above it and $b$ is the top label of the down elbow immediately to the left.  Notice that for every up-elbow in pipe $i$, the label $(a,b)$ has the properties:
\begin{itemize}
    \item [1.] $i<b<a$
    \item[2.] $\omega^{-1}(a)$, $\omega^{-1}(b) > \omega^{-1}(i)$
\end{itemize}
and that no two up-elbows from the same pipe can have labels which share the same first coordinate or the same second coordinate.  \\
\\
Let $(a_1,b_1),\ldots,(a_k,b_k)$ be the labels of all the up-elbows on pipe $i$.  By the properties above, we can sort the labels into disjoint lists of the form $(a_{j_1},b_{j_1}),\ldots,(a_{j_m},b_{j_m})$ where each $b_{j_l}=a_{j_{l+1}}$ and $a_{j_1},a_{j_2},\ldots,a_{j_m},b_{j_m}$ is a decreasing list of $m+1$ numbers which are larger than $i$ and follow $i$ in the one-line notation of the permutation.  In order to create an increasing sequence in $\omega$ beginning with $i$, we would, thus, have to remove at least $m$ elements from each such list.  This implies that we would have to remove at least $k$ such elements in total, hence, $k \leq r_{\omega^{-1}(i)}-c_{\omega^{-1}(i)}$.   \\
\\
By Lemma \ref{construction}, there exist BPDs in which every pipe has $r_{\omega^{-1}(i)}-c_{\omega^{-1}(i)}$ up-elbows, so these must be precisely the BPDs which correspond to maximal degree terms of $\mathfrak{G}_{\omega}$.
\end{proof}
\noindent Note that Theorem \ref{main} serves as an alternate proof of Theorem \ref{PSWdegree} in the special case where $\omega$ is vexillary.  Theorem \ref{main} also implies the following connection between the degree formulas given in Theorems \ref{PSWdegree} and \ref{RRWdegree}.
\begin{corollary3}
Let $\omega \in S_n$ be vexillary.  Then $\sum_{\{i : rank_{\omega}(i)=k\}}(r_{\omega^{-1}(i)}-c_{\omega^{-1}(i)})=\rho_a(\tau_k(\omega))$.
\end{corollary3}
\begin{proof}
By identifying the shaded portion of the diagrams $D_i$ with the connected components of the $\tau_k(\omega)$, this follows immediately from the proof of Lemma \ref{construction} and Theorem \ref{main}.
\end{proof}
\begin{proposition} 
\label{sameRowCol}
For a fixed vexillary permutation $\omega$, consider the set of all bumpless pipe dreams $P$ such that the marked BPD $(P,U(P))$ corresponds to a maximal degree term of $\mathfrak{G}_{\omega}$.  In the $i^{th}$ pipe, the $j^{th}$ up-elbow will either be in the same row for all maximal degree BPDs or in the same column for all maximal degree BPDs.
\end{proposition}
\begin{proof}
Let $m$ be as large as possible such that $r_{m}-c_{m}>0$.  By Lemma \ref{VexPipesOrdering}, we can construct any BPD from the Rothe BPD by first positioning pipe $\omega(m)$, then pipe $\omega(m-1)$ and so on, and we know that for any maximal degree BPD, pipe $\omega(m)$ will have precisely $r_{m}-c_{m}$ up-elbows.  Furthermore, if pipe $\omega(m)$ has $r_{m}-c_{m}$ up-elbows, the $j^{th}$ up-elbow will either be in the same row for all possible BPDs or in the same column for all possible BPDs.  To see this, consider any two BPDs $P_1$ and $P_2$ such that pipe $\omega(m)$ has $r_{m}-c_{m}$ up-elbows and (without loss of generality) no other pipes have any up-elbows.  Let $j$ be the first index such that the $j^{th}$ up-elbow of pipe $\omega(m)$ in $P_1$ is in both a different row and column from the corresponding up-elbow in $P_2$, and let $(x_1,y_1)$ and $(x_2,y_2)$ be the coordinates of this $j^{th}$ up-elbow in $P_1$ and $P_2$.  We can perform local moves on $P_1$ and $P_2$ to slide the first $j-1$ up-elbows to the lowest/leftmost position of the two without changing the remainder of the pipe (after the $j^{th}$ elbow).  Call the new BPDs $P'_1$ and $P'_2$ respectively.  Supposing without loss of generality that $x_1<x_2$, there are now two cases.
 \begin{itemize}
    \item []Case 1: $y_1>y_2$ \\ \\
    In this case, note that $P'_1$ will have an empty tile in position $(x_2,y_2)$ and the down-elbow immediately to the left of the $j^{th}$ up-elbow can be drooped into that square to form another up-elbow.  This contradicts the assumption that our pipe had the maximum number of up-elbows.  
    \item[]Case 2: $y_1<y_2$ \\ \\
    In this case, $P'_1$ will have an empty tile in positions $(x_1,y_1+1)$ and $(x_2,y_1)$.  We can use a local move to slide the $j^{th}$ up-elbow down to position $(x_2,y_1)$.  Now, if our pipe does not have an up-elbow in column $y_1+1$, we can droop the down-elbow immediately above the $j^{th}$ up-elbow into square $(x_1,y_1+1)$ giving a contradiction.  If up-elbow $j+1$ in $P'_2$ is in the same row as up-elbow $j+1$ in $P'_1$, we can up-droop the $j^{th}$ elbow of $P'_2$ to square $(x_1,y_1)$, increasing the number of up-elbows and again yielding a contradiction.  Otherwise, if none of the above cases hold, $j+1$ is now the first index where our pipe has a corresponding pair of elbows that share neither a row nor a column.  We can then repeat the above argument until one of the other cases is reached.  
 \end{itemize}
We, thus, conclude that no such up-elbow $j$ can exist. \\
\\
Now, set pipe $\omega(m)$ to the position where each up-elbow is shifted as far downward/rightward as possible and repeat the above argument for pipe $\omega(m-1)$.  Note that with pipe $\omega(m)$ positioned in this way, all positions of pipe $\omega(m-1)$ which have a maximal number of up-elbows will still be possible.  Iterating this process for all the pipes gives the result.
\end{proof}
\begin{corollary}
\label{topLocalMoves}
For a fixed vexillary permutation $\omega$, the set of all bumpless pipe dreams $P$ such that the marked BPD $(P,U(P))$ corresponds to a top degree term of $\mathfrak{G}_{\omega}$ is connected by the following local moves (and their inverses)\\
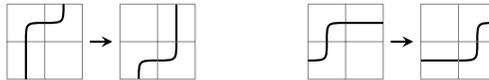
\begin{center}
\scalebox{1}{\begin{tikzpicture}
\draw[step=.5cm,gray,very thin] (0,0) grid (1,1);
\draw[black, thick] (.25,.5) .. controls (.25,.75) .. (.5,.75);
\draw[black, thick] (.5,.75) .. controls (.75,.75) .. (.75,1);
\draw[black, thick] (.25,0) -- (.25,.5);
\draw[black, thick] (1.75,0) .. controls (1.75,.25) .. (2,.25);
\draw[black, thick] (2,.25) .. controls (2.25,.25) .. (2.25,.5);
\draw[black, thick] (2.25,.5) -- (2.25,1);
\draw[step=.5cm,gray,very thin] (1.4999,0) grid (2.5,1);
\draw[-stealth, black, thick] (1.1,.5) -- (1.4,.5);

\draw[step=.5cm,gray,very thin] (3.999,0) grid (5,1);
\draw[black, thick] (4.25,.5) .. controls (4.25,.75) .. (4.5,.75);
\draw[black, thick] (4,.25) .. controls (4.25,.25) .. (4.25,.5);
\draw[black, thick] (4.5,.75) -- (5,.75);
\draw[black, thick] (6.25,.5) .. controls (6.25,.75) .. (6.5,.75);
\draw[black, thick] (6,.25) .. controls (6.25,.25) .. (6.25,.5);
\draw[black, thick] (5.5,.25) -- (6,.25);
\draw[step=.5cm,gray,very thin] (5.4999,0) grid (6.5,1);
\draw[-stealth, black, thick] (5.1,.5) -- (5.4,.5);

\end{tikzpicture}} \\
\end{center}
which do not change the total number of up-elbows.
\end{corollary}
\begin{proof}
Begin with the maximal degree bumpless pipe dream $P$ in which all the up-elbows have been pushed as far downward/rightward as possible (as constructed above).  For any other maximal degree bumpless pipe dream $P'$, we can now use the appropriate local moves to slide the up-elbows of pipe $\omega(1)$ from their positions in $P$ to their positions in $P'$.  We can then repeat this for pipes $\omega(2),\ldots,\omega(n)$.
\end{proof}
\begin{remark}
Applying a move of the form above either does not change the monomial corresponding to the marked BPD $(P,U(P))$ (in the case of the first move) or decreases the exponent of $x_i$ and increases the exponent of $x_{i-1}$ by one (in the case of the second).  Accordingly, the maximal degree bumpless pipe dream $P$ in which all the up-elbows have been pushed as far downward/rightward as possible corresponds to the leading term of $\mathfrak{G}_{\omega}$ in any term order with $x_1 > x_2 > \ldots$. 
\end{remark}
\begin{example}
Let $\omega=1275463$.  The set of BPDs $P \in Pipes(\omega)$ such that $(P,U(P))$ corresponds to a maximal degree term is shown below.  A pair of BPDs are connected by an edge if and only if they differ by precisely one of the moves from Corollary \ref{topLocalMoves}.
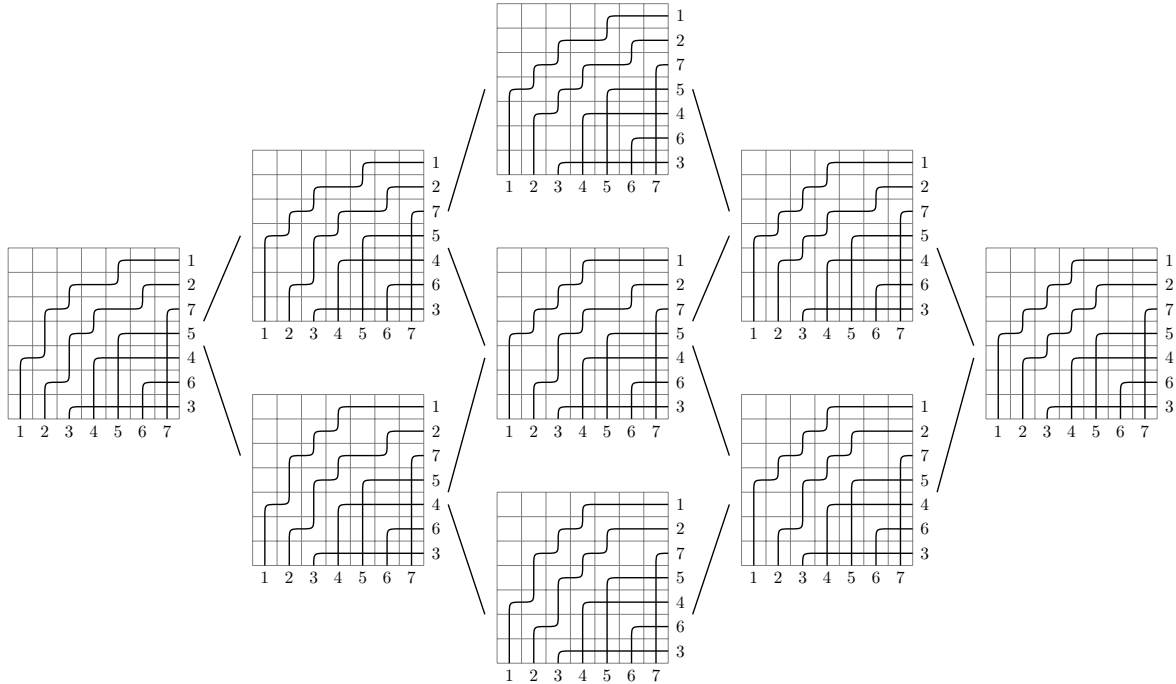
\begin{figure}[H]
    \centering
    \scalebox{.65}{\begin{tikzpicture} 
\draw[step=.5cm,gray,very thin] (0,0) grid (3.5,3.5);
\node at (.25,-.25){1};
\node at (.75,-.25){2};
\node at (1.25,-.25){3};
\node at (1.75,-.25){4};
\node at (2.25,-.25){5};
\node at (2.75,-.25){6};
\node at (3.25,-.25){7};
\node at (3.75,3.25){1};
\node at (3.75,2.75){2};
\node at (3.75,2.25){7};
\node at (3.75,1.75){5};
\node at (3.75,1.25){4};
\node at (3.75,.75){6};
\node at (3.75,.25){3};
\draw[black, thick] (3.25,0) -- (3.25,2) .. controls (3.25,2.25) .. (3.5,2.25);
\draw[black,thick] (2.75,0) -- (2.75,.5) .. controls (2.75,.75) .. (3,.75) -- (3.5,.75);
\draw[black, thick] (2.25,0) -- (2.25, 1.5) .. controls (2.25,1.75) .. (2.5,1.75) -- (3.5,1.75);
\draw[black, thick] (1.75,0) -- (1.75,1) .. controls (1.75,1.25) .. (2,1.25) -- (3.5,1.25);
\draw[black, thick] (1.25, 0) .. controls (1.25,.25) .. (1.5,.25) -- (3.5,.25);
\draw[black, thick] (.75,0) -- (.75,.5) .. controls (.75,.75) .. (1,.75) .. controls (1.25,.75) .. (1.25,1) -- (1.25,1.5) .. controls (1.25,1.75) .. (1.5,1.75) .. controls (1.75,1.75) .. (1.75,2).. controls (1.75,2.25) .. (2,2.25) -- (2.5,2.25) .. controls (2.75,2.25) .. (2.75,2.5) .. controls (2.75,2.75) .. (3,2.75) -- (3.5,2.75);
\draw[black,thick] (.25,0) -- (.25,1) .. controls (.25,1.25) .. (.5,1.25) .. controls (.75,1.25) .. (.75,1.5) -- (.75,2) .. controls (.75,2.25) .. (1,2.25) .. controls (1.25,2.25) .. (1.25,2.5) .. controls (1.25,2.75) .. (1.5,2.75) -- (2,2.75) .. controls (2.25,2.75) .. (2.25,3) .. controls (2.25,3.25) .. (2.5,3.25) -- (3.5,3.25);
\begin{scope}[shift={(5,2)}]
\draw[step=.5cm,gray,very thin] (0,0) grid (3.5,3.5);
\node at (.25,-.25){1};
\node at (.75,-.25){2};
\node at (1.25,-.25){3};
\node at (1.75,-.25){4};
\node at (2.25,-.25){5};
\node at (2.75,-.25){6};
\node at (3.25,-.25){7};
\node at (3.75,3.25){1};
\node at (3.75,2.75){2};
\node at (3.75,2.25){7};
\node at (3.75,1.75){5};
\node at (3.75,1.25){4};
\node at (3.75,.75){6};
\node at (3.75,.25){3};
\draw[black, thick] (3.25,0) -- (3.25,2) .. controls (3.25,2.25) .. (3.5,2.25);
\draw[black,thick] (2.75,0) -- (2.75,.5) .. controls (2.75,.75) .. (3,.75) -- (3.5,.75);
\draw[black, thick] (2.25,0) -- (2.25, 1.5) .. controls (2.25,1.75) .. (2.5,1.75) -- (3.5,1.75);
\draw[black, thick] (1.75,0) -- (1.75,1) .. controls (1.75,1.25) .. (2,1.25) -- (3.5,1.25);
\draw[black, thick] (1.25, 0) .. controls (1.25,.25) .. (1.5,.25) -- (3.5,.25);
\draw[black, thick] (.75,0) -- (.75,.5) .. controls (.75,.75) .. (1,.75) .. controls (1.25,.75) .. (1.25,1) -- (1.25,1.5) .. controls (1.25,1.75) .. (1.5,1.75) .. controls (1.75,1.75) .. (1.75,2).. controls (1.75,2.25) .. (2,2.25) -- (2.5,2.25) .. controls (2.75,2.25) .. (2.75,2.5) .. controls (2.75,2.75) .. (3,2.75) -- (3.5,2.75);
\draw[black,thick] (.25,0) -- (.25,1.5) .. controls (.25,1.75) .. (.5,1.75) .. controls (.75,1.75) .. (.75,2) -- (.75,2) .. controls (.75,2.25) .. (1,2.25) .. controls (1.25,2.25) .. (1.25,2.5) .. controls (1.25,2.75) .. (1.5,2.75) -- (2,2.75) .. controls (2.25,2.75) .. (2.25,3) .. controls (2.25,3.25) .. (2.5,3.25) -- (3.5,3.25);
\end{scope}
\begin{scope}[shift={(10,5)}]
\draw[step=.5cm,gray,very thin] (0,0) grid (3.5,3.5);
\node at (.25,-.25){1};
\node at (.75,-.25){2};
\node at (1.25,-.25){3};
\node at (1.75,-.25){4};
\node at (2.25,-.25){5};
\node at (2.75,-.25){6};
\node at (3.25,-.25){7};
\node at (3.75,3.25){1};
\node at (3.75,2.75){2};
\node at (3.75,2.25){7};
\node at (3.75,1.75){5};
\node at (3.75,1.25){4};
\node at (3.75,.75){6};
\node at (3.75,.25){3};
\draw[black, thick] (3.25,0) -- (3.25,2) .. controls (3.25,2.25) .. (3.5,2.25);
\draw[black,thick] (2.75,0) -- (2.75,.5) .. controls (2.75,.75) .. (3,.75) -- (3.5,.75);
\draw[black, thick] (2.25,0) -- (2.25, 1.5) .. controls (2.25,1.75) .. (2.5,1.75) -- (3.5,1.75);
\draw[black, thick] (1.75,0) -- (1.75,1) .. controls (1.75,1.25) .. (2,1.25) -- (3.5,1.25);
\draw[black, thick] (1.25, 0) .. controls (1.25,.25) .. (1.5,.25) -- (3.5,.25);
\draw[black, thick] (.75,0) -- (.75,1) .. controls (.75,1.25) .. (1,1.25) .. controls (1.25,1.25) .. (1.25,1.5) -- (1.25,1.5) .. controls (1.25,1.75) .. (1.5,1.75) .. controls (1.75,1.75) .. (1.75,2).. controls (1.75,2.25) .. (2,2.25) -- (2.5,2.25) .. controls (2.75,2.25) .. (2.75,2.5) .. controls (2.75,2.75) .. (3,2.75) -- (3.5,2.75);
\draw[black,thick] (.25,0) -- (.25,1.5) .. controls (.25,1.75) .. (.5,1.75) .. controls (.75,1.75) .. (.75,2) -- (.75,2) .. controls (.75,2.25) .. (1,2.25) .. controls (1.25,2.25) .. (1.25,2.5) .. controls (1.25,2.75) .. (1.5,2.75) -- (2,2.75) .. controls (2.25,2.75) .. (2.25,3) .. controls (2.25,3.25) .. (2.5,3.25) -- (3.5,3.25);
\end{scope}
\begin{scope}[shift={(5,-3)}]
\draw[step=.5cm,gray,very thin] (0,0) grid (3.5,3.5);
\node at (.25,-.25){1};
\node at (.75,-.25){2};
\node at (1.25,-.25){3};
\node at (1.75,-.25){4};
\node at (2.25,-.25){5};
\node at (2.75,-.25){6};
\node at (3.25,-.25){7};
\node at (3.75,3.25){1};
\node at (3.75,2.75){2};
\node at (3.75,2.25){7};
\node at (3.75,1.75){5};
\node at (3.75,1.25){4};
\node at (3.75,.75){6};
\node at (3.75,.25){3};
\draw[black, thick] (3.25,0) -- (3.25,2) .. controls (3.25,2.25) .. (3.5,2.25);
\draw[black,thick] (2.75,0) -- (2.75,.5) .. controls (2.75,.75) .. (3,.75) -- (3.5,.75);
\draw[black, thick] (2.25,0) -- (2.25, 1.5) .. controls (2.25,1.75) .. (2.5,1.75) -- (3.5,1.75);
\draw[black, thick] (1.75,0) -- (1.75,1) .. controls (1.75,1.25) .. (2,1.25) -- (3.5,1.25);
\draw[black, thick] (1.25, 0) .. controls (1.25,.25) .. (1.5,.25) -- (3.5,.25);
\draw[black, thick] (.75,0) -- (.75,.5) .. controls (.75,.75) .. (1,.75) .. controls (1.25,.75) .. (1.25,1) -- (1.25,1.5) .. controls (1.25,1.75) .. (1.5,1.75) .. controls (1.75,1.75) .. (1.75,2).. controls (1.75,2.25) .. (2,2.25) -- (2.5,2.25) .. controls (2.75,2.25) .. (2.75,2.5) .. controls (2.75,2.75) .. (3,2.75) -- (3.5,2.75);
\draw[black,thick] (.25,0) -- (.25,1) .. controls (.25,1.25) .. (.5,1.25) .. controls (.75,1.25) .. (.75,1.5) -- (.75,2) .. controls (.75,2.25) .. (1,2.25) .. controls (1.25,2.25) .. (1.25,2.5) .. controls (1.25,2.75) .. (1.5,2.75) -- (1.5,2.75) .. controls (1.75,2.75) .. (1.75,3) .. controls (1.75,3.25) .. (2,3.25) -- (3.5,3.25);
\end{scope}
\begin{scope}[shift={(10,-5)}]
\draw[step=.5cm,gray,very thin] (0,0) grid (3.5,3.5);
\node at (.25,-.25){1};
\node at (.75,-.25){2};
\node at (1.25,-.25){3};
\node at (1.75,-.25){4};
\node at (2.25,-.25){5};
\node at (2.75,-.25){6};
\node at (3.25,-.25){7};
\node at (3.75,3.25){1};
\node at (3.75,2.75){2};
\node at (3.75,2.25){7};
\node at (3.75,1.75){5};
\node at (3.75,1.25){4};
\node at (3.75,.75){6};
\node at (3.75,.25){3};
\draw[black, thick] (3.25,0) -- (3.25,2) .. controls (3.25,2.25) .. (3.5,2.25);
\draw[black,thick] (2.75,0) -- (2.75,.5) .. controls (2.75,.75) .. (3,.75) -- (3.5,.75);
\draw[black, thick] (2.25,0) -- (2.25, 1.5) .. controls (2.25,1.75) .. (2.5,1.75) -- (3.5,1.75);
\draw[black, thick] (1.75,0) -- (1.75,1) .. controls (1.75,1.25) .. (2,1.25) -- (3.5,1.25);
\draw[black, thick] (1.25, 0) .. controls (1.25,.25) .. (1.5,.25) -- (3.5,.25);
\draw[black, thick] (.75,0) -- (.75,.5) .. controls (.75,.75) .. (1,.75) .. controls (1.25,.75) .. (1.25,1) -- (1.25,1.5) .. controls (1.25,1.75) .. (1.5,1.75) .. controls (1.75,1.75) .. (1.75,2).. controls (1.75,2.25) .. (2,2.25) -- (2,2.25) .. controls (2.25,2.25) .. (2.25,2.5) .. controls (2.25,2.75) .. (2.5,2.75) -- (3.5,2.75);
\draw[black,thick] (.25,0) -- (.25,1) .. controls (.25,1.25) .. (.5,1.25) .. controls (.75,1.25) .. (.75,1.5) -- (.75,2) .. controls (.75,2.25) .. (1,2.25) .. controls (1.25,2.25) .. (1.25,2.5) .. controls (1.25,2.75) .. (1.5,2.75) -- (1.5,2.75) .. controls (1.75,2.75) .. (1.75,3) .. controls (1.75,3.25) .. (2,3.25) -- (3.5,3.25);
\end{scope}
\begin{scope}[shift={(10,0)}]
\draw[step=.5cm,gray,very thin] (0,0) grid (3.5,3.5);
\node at (.25,-.25){1};
\node at (.75,-.25){2};
\node at (1.25,-.25){3};
\node at (1.75,-.25){4};
\node at (2.25,-.25){5};
\node at (2.75,-.25){6};
\node at (3.25,-.25){7};
\node at (3.75,3.25){1};
\node at (3.75,2.75){2};
\node at (3.75,2.25){7};
\node at (3.75,1.75){5};
\node at (3.75,1.25){4};
\node at (3.75,.75){6};
\node at (3.75,.25){3};
\draw[black, thick] (3.25,0) -- (3.25,2) .. controls (3.25,2.25) .. (3.5,2.25);
\draw[black,thick] (2.75,0) -- (2.75,.5) .. controls (2.75,.75) .. (3,.75) -- (3.5,.75);
\draw[black, thick] (2.25,0) -- (2.25, 1.5) .. controls (2.25,1.75) .. (2.5,1.75) -- (3.5,1.75);
\draw[black, thick] (1.75,0) -- (1.75,1) .. controls (1.75,1.25) .. (2,1.25) -- (3.5,1.25);
\draw[black, thick] (1.25, 0) .. controls (1.25,.25) .. (1.5,.25) -- (3.5,.25);
\draw[black, thick] (.75,0) -- (.75,.5) .. controls (.75,.75) .. (1,.75) .. controls (1.25,.75) .. (1.25,1) -- (1.25,1.5) .. controls (1.25,1.75) .. (1.5,1.75) .. controls (1.75,1.75) .. (1.75,2).. controls (1.75,2.25) .. (2,2.25) -- (2.5,2.25) .. controls (2.75,2.25) .. (2.75,2.5) .. controls (2.75,2.75) .. (3,2.75) -- (3.5,2.75);
\draw[black,thick] (.25,0) -- (.25,1.5) .. controls (.25,1.75) .. (.5,1.75) .. controls (.75,1.75) .. (.75,2) -- (.75,2) .. controls (.75,2.25) .. (1,2.25) .. controls (1.25,2.25) .. (1.25,2.5) .. controls (1.25,2.75) .. (1.5,2.75) -- (1.5,2.75) .. controls (1.75,2.75) .. (1.75,3) .. controls (1.75,3.25) .. (2,3.25) -- (3.5,3.25);
\end{scope}
\begin{scope}[shift={(15,2)}]
\draw[step=.5cm,gray,very thin] (0,0) grid (3.5,3.5);
\node at (.25,-.25){1};
\node at (.75,-.25){2};
\node at (1.25,-.25){3};
\node at (1.75,-.25){4};
\node at (2.25,-.25){5};
\node at (2.75,-.25){6};
\node at (3.25,-.25){7};
\node at (3.75,3.25){1};
\node at (3.75,2.75){2};
\node at (3.75,2.25){7};
\node at (3.75,1.75){5};
\node at (3.75,1.25){4};
\node at (3.75,.75){6};
\node at (3.75,.25){3};
\draw[black, thick] (3.25,0) -- (3.25,2) .. controls (3.25,2.25) .. (3.5,2.25);
\draw[black,thick] (2.75,0) -- (2.75,.5) .. controls (2.75,.75) .. (3,.75) -- (3.5,.75);
\draw[black, thick] (2.25,0) -- (2.25, 1.5) .. controls (2.25,1.75) .. (2.5,1.75) -- (3.5,1.75);
\draw[black, thick] (1.75,0) -- (1.75,1) .. controls (1.75,1.25) .. (2,1.25) -- (3.5,1.25);
\draw[black, thick] (1.25, 0) .. controls (1.25,.25) .. (1.5,.25) -- (3.5,.25);
\draw[black, thick] (.75,0) -- (.75,1) .. controls (.75,1.25) .. (1,1.25) .. controls (1.25,1.25) .. (1.25,1.5) -- (1.25,1.5) .. controls (1.25,1.75) .. (1.5,1.75) .. controls (1.75,1.75) .. (1.75,2).. controls (1.75,2.25) .. (2,2.25) -- (2.5,2.25) .. controls (2.75,2.25) .. (2.75,2.5) .. controls (2.75,2.75) .. (3,2.75) -- (3.5,2.75);
\draw[black,thick] (.25,0) -- (.25,1.5) .. controls (.25,1.75) .. (.5,1.75) .. controls (.75,1.75) .. (.75,2) -- (.75,2) .. controls (.75,2.25) .. (1,2.25) .. controls (1.25,2.25) .. (1.25,2.5) .. controls (1.25,2.75) .. (1.5,2.75) -- (1.5,2.75) .. controls (1.75,2.75) .. (1.75,3) .. controls (1.75,3.25) .. (2,3.25) -- (3.5,3.25);
\end{scope}
\begin{scope}[shift={(15,-3)}]
\draw[step=.5cm,gray,very thin] (0,0) grid (3.5,3.5);
\node at (.25,-.25){1};
\node at (.75,-.25){2};
\node at (1.25,-.25){3};
\node at (1.75,-.25){4};
\node at (2.25,-.25){5};
\node at (2.75,-.25){6};
\node at (3.25,-.25){7};
\node at (3.75,3.25){1};
\node at (3.75,2.75){2};
\node at (3.75,2.25){7};
\node at (3.75,1.75){5};
\node at (3.75,1.25){4};
\node at (3.75,.75){6};
\node at (3.75,.25){3};
\draw[black, thick] (3.25,0) -- (3.25,2) .. controls (3.25,2.25) .. (3.5,2.25);
\draw[black,thick] (2.75,0) -- (2.75,.5) .. controls (2.75,.75) .. (3,.75) -- (3.5,.75);
\draw[black, thick] (2.25,0) -- (2.25, 1.5) .. controls (2.25,1.75) .. (2.5,1.75) -- (3.5,1.75);
\draw[black, thick] (1.75,0) -- (1.75,1) .. controls (1.75,1.25) .. (2,1.25) -- (3.5,1.25);
\draw[black, thick] (1.25, 0) .. controls (1.25,.25) .. (1.5,.25) -- (3.5,.25);
\draw[black, thick] (.75,0) -- (.75,.5) .. controls (.75,.75) .. (1,.75) .. controls (1.25,.75) .. (1.25,1) -- (1.25,1.5) .. controls (1.25,1.75) .. (1.5,1.75) .. controls (1.75,1.75) .. (1.75,2).. controls (1.75,2.25) .. (2,2.25) -- (2,2.25) .. controls (2.25,2.25) .. (2.25,2.5) .. controls (2.25,2.75) .. (2.5,2.75) -- (3.5,2.75);
\draw[black,thick] (.25,0) -- (.25,1.5) .. controls (.25,1.75) .. (.5,1.75) .. controls (.75,1.75) .. (.75,2) -- (.75,2) .. controls (.75,2.25) .. (1,2.25) .. controls (1.25,2.25) .. (1.25,2.5) .. controls (1.25,2.75) .. (1.5,2.75) -- (1.5,2.75) .. controls (1.75,2.75) .. (1.75,3) .. controls (1.75,3.25) .. (2,3.25) -- (3.5,3.25);
\end{scope}
\begin{scope}[shift={(20,0)}]
\draw[step=.5cm,gray,very thin] (0,0) grid (3.5,3.5);
\node at (.25,-.25){1};
\node at (.75,-.25){2};
\node at (1.25,-.25){3};
\node at (1.75,-.25){4};
\node at (2.25,-.25){5};
\node at (2.75,-.25){6};
\node at (3.25,-.25){7};
\node at (3.75,3.25){1};
\node at (3.75,2.75){2};
\node at (3.75,2.25){7};
\node at (3.75,1.75){5};
\node at (3.75,1.25){4};
\node at (3.75,.75){6};
\node at (3.75,.25){3};
\draw[black, thick] (3.25,0) -- (3.25,2) .. controls (3.25,2.25) .. (3.5,2.25);
\draw[black,thick] (2.75,0) -- (2.75,.5) .. controls (2.75,.75) .. (3,.75) -- (3.5,.75);
\draw[black, thick] (2.25,0) -- (2.25, 1.5) .. controls (2.25,1.75) .. (2.5,1.75) -- (3.5,1.75);
\draw[black, thick] (1.75,0) -- (1.75,1) .. controls (1.75,1.25) .. (2,1.25) -- (3.5,1.25);
\draw[black, thick] (1.25, 0) .. controls (1.25,.25) .. (1.5,.25) -- (3.5,.25);
\draw[black, thick] (.75,0) -- (.75,1) .. controls (.75,1.25) .. (1,1.25) .. controls (1.25,1.25) .. (1.25,1.5) -- (1.25,1.5) .. controls (1.25,1.75) .. (1.5,1.75) .. controls (1.75,1.75) .. (1.75,2).. controls (1.75,2.25) .. (2,2.25) -- (2,2.25) .. controls (2.25,2.25) .. (2.25,2.5) .. controls (2.25,2.75) .. (2.5,2.75) -- (3.5,2.75);
\draw[black,thick] (.25,0) -- (.25,1.5) .. controls (.25,1.75) .. (.5,1.75) .. controls (.75,1.75) .. (.75,2) -- (.75,2) .. controls (.75,2.25) .. (1,2.25) .. controls (1.25,2.25) .. (1.25,2.5) .. controls (1.25,2.75) .. (1.5,2.75) -- (1.5,2.75) .. controls (1.75,2.75) .. (1.75,3) .. controls (1.75,3.25) .. (2,3.25) -- (3.5,3.25);
\end{scope}
\draw[black,thick] (4,2) -- (4.75,3.75);
\draw[black,thick] (4,1.5) -- (4.75,-.75);
\draw[black,thick] (9,4.25) -- (9.75,6.75);
\draw[black,thick] (9,3.5) -- (9.75,1.5);
\draw[black,thick] (9,-1.5) -- (9.75,1.25);
\draw[black,thick] (9,-1.75) -- (9.75,-4);
\draw[black,thick] (14,2) -- (14.75,3.75);
\draw[black,thick] (14,1.5) -- (14.75,-.75);
\draw[black,thick] (14,6.75) -- (14.75,4.25);
\draw[black,thick] (14,-4) -- (14.75,-1.75);
\draw[black,thick] (19,3.5) -- (19.75,1.5);
\draw[black,thick] (19,-1.5) -- (19.75,1.25);
\end{tikzpicture}}
    \label{fig:connectedexample}
\end{figure} 
\end{example}
\noindent We conjecture the following generalization of Theorem \ref{main} for all permutations.  \\ 
\begin{conjecture}
\label{conjElbows}
Let $\omega$ be any permutation, and let $P \in Pipes(\omega)$.  If $P$ is non-reduced, replace the redundant crossings with bump tiles (as shown in Figure \ref{fig:nonreduced}).  Pipe $i$ can have at most $r_{\omega^{-1}(i)}-c_{\omega^{-1}(i)}$ up-elbows, including the up-elbow portions of any bump tiles.
\end{conjecture}
\noindent This has been verified by computer for all permutations in $S_6$; however, the technique from the proof of Theorem \ref{main} does not apply, even for reduced BPDs corresponding to non-vexillary permutations.  \\ \\

\section{Support of Vexillary Grothendieck Polynomials} 
\noindent In this section, we address some properties of vexillary Grothendieck polynomials which can be proved using bumpless pipe dreams.  We begin with the vexillary case of Conjecture \ref{conjsides}. \\  
\begin{proposition}
\label{upByOne}
Let $\omega \in S_n$ be a vexillary permutation, and let $x_1^{i_1}x_2^{i_2}\ldots x_n^{i_n}$ be any monomial with a nonzero coefficient and non-maximal degree in $\mathfrak{G}_{\omega}(x_1,\ldots,x_n)$.  There exists a monomial $x_1^{j_1}x_2^{j_2}\ldots x_n^{j_n}$ with nonzero coefficient in $\mathfrak{G}_{\omega}(x_1,\ldots,x_n)$ such that $j_k=i_k+1$ for some $1 \leq k \leq n$ and $j_l=i_l$ for all other indices $l$.
\end{proposition}
\begin{proof}
Let $(P,S) \in MPipes(\omega)$ be a marked bumpless pipe dream corresponding to the monomial $x_1^{i_1}x_2^{i_2}\ldots x_n^{i_n}$.  If $S \neq U(P)$, then there is some up-elbow tile $(i,j)$ in $P$ which is not in $S$, and the marked BPD $(P,S \cup \{(i,j)\})$ will correspond to a monomial $x_1^{j_1}x_2^{j_2}\ldots x_n^{j_n}$ with the desired properties.\\
\\
Otherwise, if $x_1^{i_1}x_2^{i_2}\ldots x_n^{i_n}$ corresponds to a marked BPD $(P,U(P))$, there must be at least one pipe in $P$ which has fewer than $r_{\omega^{-1}(k)}-c_{\omega^{-1}(k)}$ up-elbows (since the monomial has non-maximal degree).  Select the largest $k$ such that pipe $k$ has fewer than $r_{\omega^{-1}(k)}-c_{\omega^{-1}(k)}$ up-elbows.  Note that by Theorem \ref{main}, there must be some BPD in $Pipes(\omega)$ in which pipe $k$ has more up-elbows than in $P$.  \\ \\
Define a BPD $P' \in Pipes(\omega)$ by leaving pipes $k$ through $n$ in their positions from $P$ and resetting pipes $1$ through $k-1$ to their positions from the Rothe BPD for $\omega$.  By Lemma \ref{VexPipesOrdering}, every pipe which contains any elbow tiles southeast of pipe $k$ must have the maximal number of up-elbows, so since pipe $k$ does not have the maximal number of elbow tiles, we can perform a move of one the following forms on pipe $k$ in the BPD $P'$:\\
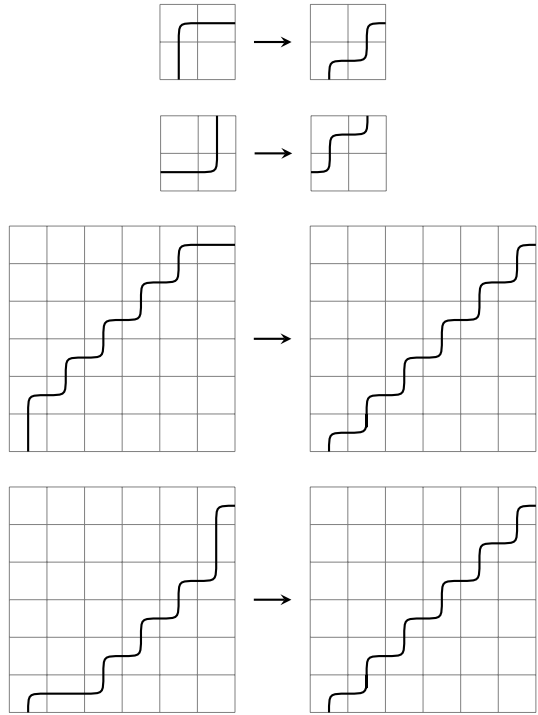
\begin{center}
    \scalebox{1}{\begin{tikzpicture}
\draw[step=.5cm,gray,very thin] (0,0) grid (1,1);
\draw[black,thick] (.25,0) -- (.25,.5) .. controls (.25,.75) .. (.5,.75) -- (1,.75);
\draw[-stealth, black, thick] (1.25,.5) -- (1.75,.5);
\draw[step=.5cm,gray,very thin] (1.999,0) grid (3,1);
\draw[black,thick] (2.25,0) .. controls (2.25,.25) .. (2.5,.25) .. controls (2.75,.25) .. (2.75,.5) .. controls (2.75,.75) .. (3,.75);
\end{tikzpicture}}

\scalebox{1}{\begin{tikzpicture}
\draw[step=.5cm,gray,very thin] (0,0) grid (1,1);
\draw[black,thick] (0,0.25) -- (0.5,0.25).. controls (0.75,0.25) .. (0.75,0.5) -- (.75, 1);
\draw[-stealth, black, thick] (1.25,.5) -- (1.75,.5);
\draw[step=.5cm,gray,very thin] (1.999,0) grid (3,1);
\draw[black,thick] (2,.25) .. controls (2.25,.25) .. (2.25,.5) .. controls (2.25,.75) .. (2.5,.75) .. controls (2.75,.75) .. (2.75,1);
\end{tikzpicture}}

\scalebox{1}{\begin{tikzpicture}
\draw[step=.5cm,gray,very thin] (0,0) grid (3,3);
\draw[black,thick] (.25,0) -- (.25,.5) .. controls (.25,.75) .. (.5,.75) .. controls (.75,.75) .. (.75,1) .. controls (.75,1.25) ..  (1,1.25) .. controls (1.25,1.25) .. (1.25,1.5) .. controls (1.25,1.75) .. (1.5,1.75) .. controls (1.75,1.75) .. (1.75,2) .. controls (1.75,2.25) .. (2,2.25) .. controls (2.25,2.25) .. (2.25,2.5) .. controls (2.25,2.75) .. (2.5,2.75) -- (3,2.75);
\draw[-stealth, black, thick] (3.25,1.5) -- (3.75,1.5);
\draw[step=.5cm,gray,very thin] (3.999,0) grid (7,3);
\draw[black,thick] (4.25,0) .. controls (4.25,.25) .. (4.5,.25) .. controls (4.75,.25) .. (4.75,.5) .. controls (4.75,.25) .. (4.75,.5) .. controls (4.75,.75) .. (5,.75) .. controls (5.25,.75) .. (5.25, 1) .. controls (5.25,1.25) .. (5.5,1.25) .. controls (5.75,1.25) .. (5.75,1.5) .. controls (5.75,1.75) .. (6.,1.75) .. controls (6.25,1.75) .. (6.25,2) .. controls (6.25,2.25) .. (6.5,2.25).. controls (6.75,2.25) .. (6.75,2.5) .. controls (6.75,2.75) .. (7,2.75);
\end{tikzpicture}}

\scalebox{1}{\begin{tikzpicture}
\draw[step=.5cm,gray,very thin] (0,0) grid (3,3);
\draw[black,thick] (.25,0) .. controls (.25,.25) .. (.5,.25) -- (1,.25) .. controls (1.25,.25) .. (1.25,.5) .. controls (1.25,.75) .. (1.5,.75) .. controls (1.75,.75) .. (1.75, 1) .. controls (1.75,1.25) .. (2,1.25) .. controls (2.25,1.25) .. (2.25,1.5) .. controls (2.25,1.75) .. (2.5,1.75) .. controls (2.75,1.75) .. (2.75,2) -- (2.75,2.5) .. controls (2.75,2.75) .. (3,2.75);
\draw[-stealth, black, thick] (3.25,1.5) -- (3.75,1.5);
\draw[step=.5cm,gray,very thin] (3.999,0) grid (7,3);
\draw[black,thick] (4.25,0) .. controls (4.25,.25) .. (4.5,.25) .. controls (4.75,.25) .. (4.75,.5) .. controls (4.75,.25) .. (4.75,.5) .. controls (4.75,.75) .. (5,.75) .. controls (5.25,.75) .. (5.25, 1) .. controls (5.25,1.25) .. (5.5,1.25) .. controls (5.75,1.25) .. (5.75,1.5) .. controls (5.75,1.75) .. (6.,1.75) .. controls (6.25,1.75) .. (6.25,2) .. controls (6.25,2.25) .. (6.5,2.25).. controls (6.75,2.25) .. (6.75,2.5) .. controls (6.75,2.75) .. (7,2.75);

\end{tikzpicture}}
\\
\end{center}
i.e. a local droop move, a reverse local droop move, or several repetitions of one of the following local moves \\
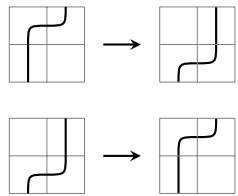
\begin{center}
        \scalebox{1}{\begin{tikzpicture}
\draw[step=.5cm,gray,very thin] (0,0) grid (1,1);
\draw[black, thick] (.25,.5) .. controls (.25,.75) .. (.5,.75);
\draw[black, thick] (.5,.75) .. controls (.75,.75) .. (.75,1);
\draw[black, thick] (.25,0) -- (.25,.5);
\draw[black, thick] (2.25,0) .. controls (2.25,.25) .. (2.5,.25);
\draw[black, thick] (2.5,.25) .. controls (2.75,.25) .. (2.75,.5);
\draw[black, thick] (2.75,.5) -- (2.75,1);
\draw[step=.5cm,gray,very thin] (1.9999,0) grid (3,1);
\draw[-stealth, black, thick] (1.25,.5) -- (1.75,.5);
\end{tikzpicture}}

\scalebox{1}{\begin{tikzpicture}
\draw[step=.5cm,gray,very thin] (0,0) grid (1,1);
\draw[black, thick] (.25,0) .. controls (.25,.25) .. (.5,.25) .. controls (.75,.25) .. (.75,.5) -- (.75,1);
\draw[black,thick] (2.25,0) -- (2.25,.5) .. controls (2.25,.75) .. (2.5,.75) .. controls (2.75,.75) .. (2.75,1);
\draw[step=.5cm,gray,very thin] (1.9999,0) grid (3,1);
\draw[-stealth, black, thick] (1.25,.5) -- (1.75,.5);
\end{tikzpicture}}
\end{center}
followed by a local droop or a reverse local droop.  Now, we perform the equivalent move on pipe $k$ in our original BPD $P$.  If the output is a valid BPD, we will denote the new BPD as $P_1$.  Otherwise, we must have performed one of the following two types of moves:\\
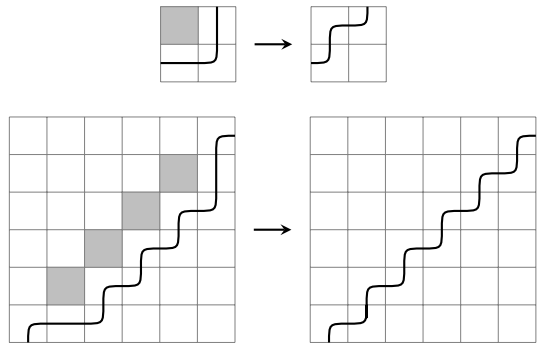
\begin{center}
    \scalebox{1}{\begin{tikzpicture}
\filldraw[draw=gray, color=lightgray] (0,.5) rectangle (.5,1);
\draw[step=.5cm,gray,very thin] (0,0) grid (1,1);
\draw[black,thick] (0,0.25) -- (0.5,0.25).. controls (0.75,0.25) .. (0.75,0.5) -- (.75, 1);
\draw[-stealth, black, thick] (1.25,.5) -- (1.75,.5);
\draw[step=.5cm,gray,very thin] (1.999,0) grid (3,1);
\draw[black,thick] (2,.25) .. controls (2.25,.25) .. (2.25,.5) .. controls (2.25,.75) .. (2.5,.75) .. controls (2.75,.75) .. (2.75,1);
\end{tikzpicture}}

\scalebox{1}{\begin{tikzpicture}
\filldraw[draw=gray, color=lightgray] (.5,.5) rectangle (1,1);
\filldraw[draw=gray, color=lightgray] (1,1) rectangle (1.5,1.5);
\filldraw[draw=gray, color=lightgray] (1.5,1.5) rectangle (2,2);
\filldraw[draw=gray, color=lightgray] (2,2) rectangle (2.5,2.5);
\draw[step=.5cm,gray,very thin] (0,0) grid (3,3);
\draw[black,thick] (.25,0) .. controls (.25,.25) .. (.5,.25) -- (1,.25) .. controls (1.25,.25) .. (1.25,.5) .. controls (1.25,.75) .. (1.5,.75) .. controls (1.75,.75) .. (1.75, 1) .. controls (1.75,1.25) .. (2,1.25) .. controls (2.25,1.25) .. (2.25,1.5) .. controls (2.25,1.75) .. (2.5,1.75) .. controls (2.75,1.75) .. (2.75,2) -- (2.75,2.5) .. controls (2.75,2.75) .. (3,2.75);
\draw[-stealth, black, thick] (3.25,1.5) -- (3.75,1.5);
\draw[step=.5cm,gray,very thin] (3.999,0) grid (7,3);
\draw[black,thick] (4.25,0) .. controls (4.25,.25) .. (4.5,.25) .. controls (4.75,.25) .. (4.75,.5) .. controls (4.75,.25) .. (4.75,.5) .. controls (4.75,.75) .. (5,.75) .. controls (5.25,.75) .. (5.25, 1) .. controls (5.25,1.25) .. (5.5,1.25) .. controls (5.75,1.25) .. (5.75,1.5) .. controls (5.75,1.75) .. (6.,1.75) .. controls (6.25,1.75) .. (6.25,2) .. controls (6.25,2.25) .. (6.5,2.25).. controls (6.75,2.25) .. (6.75,2.5) .. controls (6.75,2.75) .. (7,2.75);

\end{tikzpicture}} \\
\end{center}
where at least one of the gray squares must contain an up-elbow from some other pipe $k'$.  Furthermore, since $\omega$ is vexillary (and, thus, $Pipes(\omega)$ is connected by local moves), all up-elbows which are in the gray squares must belong to a single pipe.  Consider the rightmost grey square containing one of these up-elbows from pipe $k'$, and note that the tile immediately above it must be a vertical pipe.  We can, therefore, perform one of the following local moves \\
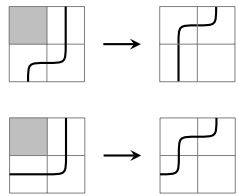
\begin{center}
    \scalebox{1}{\begin{tikzpicture}
    \filldraw[draw=gray, color=lightgray] (0,.5) rectangle (.5,1);
\draw[step=.5cm,gray,very thin] (0,0) grid (1,1);
\draw[black, thick] (.25,0) .. controls (.25,.25) .. (.5,.25) .. controls (.75,.25) .. (.75,.5) -- (.75,1);
\draw[black,thick] (2.25,0) -- (2.25,.5) .. controls (2.25,.75) .. (2.5,.75) .. controls (2.75,.75) .. (2.75,1);
\draw[step=.5cm,gray,very thin] (1.9999,0) grid (3,1);
\draw[-stealth, black, thick] (1.25,.5) -- (1.75,.5);
\end{tikzpicture}}

\scalebox{1}{\begin{tikzpicture}
\filldraw[draw=gray, color=lightgray] (0,.5) rectangle (.5,1);
\draw[step=.5cm,gray,very thin] (0,0) grid (1,1);
\draw[black,thick] (0,0.25) -- (0.5,0.25).. controls (0.75,0.25) .. (0.75,0.5) -- (.75, 1);
\draw[-stealth, black, thick] (1.25,.5) -- (1.75,.5);
\draw[step=.5cm,gray,very thin] (1.999,0) grid (3,1);
\draw[black,thick] (2,.25) .. controls (2.25,.25) .. (2.25,.5) .. controls (2.25,.75) .. (2.5,.75) .. controls (2.75,.75) .. (2.75,1);
\end{tikzpicture}} \\
\end{center}
where the shaded square may or may not contain a tile from another pipe, in order to move our rightmost shaded up-elbow out of the way.  We can similarly perform one of the above two moves on pipe $k'$ to move the next shaded up-elbow from the right and repeat this for each shaded up-elbow in pipe $k'$.\\
\\
Again, if the output is a valid BPD, we will denote it $P_1$.  Otherwise, at least one of the moves we performed on pipe $k'$ must have moved it into a tile which was already occupied by an up-elbow from a different pipe, and we repeat the process above until our output is indeed a BPD $P_1$.\\
\\
Note that moves of the forms \\
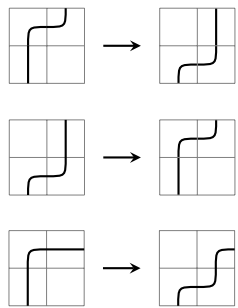
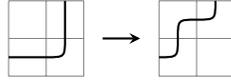
\begin{center}
    \scalebox{1}{\begin{tikzpicture}
\draw[step=.5cm,gray,very thin] (0,0) grid (1,1);
\draw[black, thick] (.25,.5) .. controls (.25,.75) .. (.5,.75);
\draw[black, thick] (.5,.75) .. controls (.75,.75) .. (.75,1);
\draw[black, thick] (.25,0) -- (.25,.5);
\draw[black, thick] (2.25,0) .. controls (2.25,.25) .. (2.5,.25);
\draw[black, thick] (2.5,.25) .. controls (2.75,.25) .. (2.75,.5);
\draw[black, thick] (2.75,.5) -- (2.75,1);
\draw[step=.5cm,gray,very thin] (1.9999,0) grid (3,1);
\draw[-stealth, black, thick] (1.25,.5) -- (1.75,.5);
\end{tikzpicture}}

\scalebox{1}{\begin{tikzpicture}
\draw[step=.5cm,gray,very thin] (0,0) grid (1,1);
\draw[black, thick] (.25,0) .. controls (.25,.25) .. (.5,.25) .. controls (.75,.25) .. (.75,.5) -- (.75,1);
\draw[black,thick] (2.25,0) -- (2.25,.5) .. controls (2.25,.75) .. (2.5,.75) .. controls (2.75,.75) .. (2.75,1);
\draw[step=.5cm,gray,very thin] (1.9999,0) grid (3,1);
\draw[-stealth, black, thick] (1.25,.5) -- (1.75,.5);
\end{tikzpicture}}

\scalebox{1}{\begin{tikzpicture}
\draw[step=.5cm,gray,very thin] (0,0) grid (1,1);
\draw[black,thick] (.25,0) -- (.25,.5) .. controls (.25,.75) .. (.5,.75) -- (1,.75);
\draw[-stealth, black, thick] (1.25,.5) -- (1.75,.5);
\draw[step=.5cm,gray,very thin] (1.999,0) grid (3,1);
\draw[black,thick] (2.25,0) .. controls (2.25,.25) .. (2.5,.25) .. controls (2.75,.25) .. (2.75,.5) .. controls (2.75,.75) .. (3,.75);
\end{tikzpicture}}

\scalebox{1}{\begin{tikzpicture}
\draw[step=.5cm,gray,very thin] (0,0) grid (1,1);
\draw[black,thick] (0,0.25) -- (0.5,0.25).. controls (0.75,0.25) .. (0.75,0.5) -- (.75, 1);
\draw[-stealth, black, thick] (1.25,.5) -- (1.75,.5);
\draw[step=.5cm,gray,very thin] (1.999,0) grid (3,1);
\draw[black,thick] (2,.25) .. controls (2.25,.25) .. (2.25,.5) .. controls (2.25,.75) .. (2.5,.75) .. controls (2.75,.75) .. (2.75,1);
\end{tikzpicture}} \\
\end{center}
either do not change the associated monomial in $\mathfrak{G}_{\omega}(x_1,\ldots,x_n)$ or add 1 to exactly one of the exponents without changing the others.  The monomial associated to $(P_1, U(P_1))$ is, thus, divisible by the original monomial.\\
\\
Now, suppose the $l^{th}$ pipe is the last pipe on which we performed either a local droop or reverse local droop move in the procedure above.  We can construct another BPD $P_2$ by starting with $P_1$ and undoing all moves performed on pipes $l' > l$ as well as all but one droop or reverse droop from pipe $l$.  $P_2$, thus, differs from $P$ by exactly one droop / reverse droop along with moves which did not affect the corresponding monomial, and $(P_2, U(P_2))$ corresponds to the desired monomial $x_1^{j_1}x_2^{j_2}\ldots x_n^{j_n}$.
\end{proof}
\noindent The following similar result is well known for Grothendieck polynomials corresponding to all permutations.  For reference, we include a new proof using BPDs.\\
\begin{lemma}
\label{downbyone}
Let $\omega \in S_n$ be \textit{any} permutation, and let $x_1^{i_1}x_2^{i_2}\ldots x_n^{i_n}$ be any monomial with a nonzero coefficient and non-minimal degree in $\mathfrak{G}_{\omega}(x_1,\ldots ,x_n)$.  There exists a monomial $x_1^{j_1}x_2^{j_2}\ldots x_n^{j_n}$ with nonzero coefficient in $\mathfrak{G}_{\omega}(x_1,\ldots,x_n)$ such that $j_k=i_k-1$ for some $1 \leq k \leq n$ and $j_l=i_l$ for all other $l$.
\end{lemma}
\begin{proof}
Let $(P,S) \in MPipes (\omega)$ be a marked bumpless pipe dream corresponding to the monomial $x_1^{i_1}x_2^{i_2}\ldots x_n^{i_n}$.  If $S$ is non-empty, then $(P,S - \{(a,b)\})$ corresponds to a monomial with the desired property for any $(a,b) \in S$.  We will, therefore, assume that $S = \emptyset$.  Note that if $P$ is reduced, then $(P, \emptyset)$ will correspond to a monomial with minimal degree, so we can also assume that $P$ is non-reduced.\\
\\
By definition, this means that there is some pair of pipes in $P$ which cross each other more than once.  Select some such pair, and replace the topmost crossing between them with a bump tile.  Let $(i_1,j_1)$ be the position of this bump tile, and let $j_2$ be the closest column to the left of $j_1$ such that the up-elbow portion of the bump tile could be reverse drooped into some blank or up-elbow tile in column $j_2$.  If the pipe can be reverse drooped into a blank square $(i_2,j_2)$, then the resulting BPD $P_2$ will have the same number of empty squares as $P$ in every row besides $i_2$ and one fewer in row $i_2$.  Otherwise, if we reverse droop the pipe into an up-elbow tile, then the above process can be repeated with the new bump tile until a valid bumpless pipe dream is reached.  Again, denote the resulting BPD $P_2$, and note that $P_2$ will still have one empty square fewer than $P$ in some row $k$ and the same number of empty squares as $P$ in all other rows.\\
\\
We thus have a marked BPD $(P_2, \emptyset) \in MPipes(\omega)$ which corresponds to a monomial with the desired property.
\end{proof}
\begin{theorem}
\label{VexSides}
Let $f(x_1,\ldots,x_n)$ be any monomial in $\mathfrak{G}_{\omega}(x_1,\ldots,x_n)$ for some vexillary permutation $\omega$.  
\begin{itemize}
    \item [(i)] There exists a monomial $f_1$ in $\mathfrak{G}_{\omega}$ such that $f_1$ has minimal degree and $f_1 | f$.
    \item [(ii)] There exists a monomial $f_2$ in $\mathfrak{G}_{\omega}$ such that $f_2$ has maximal degree and $f | f_2$.
\end{itemize}
\end{theorem}
\begin{proof}
Follows by induction from Proposition \ref{upByOne} and Lemma \ref{downbyone}.
\end{proof}
\noindent We also resolve Conjecture \ref{conjbetween} in the case of vexillary permutations.  \\
\begin{theorem}
\label{VexBetween}
Let $\omega$ be a vexillary permutation, and suppose that $p_1(x_1, \ldots, x_n)$ and $p_2(x_1, \ldots, x_n)$ are monomials with nonzero coeffiecient in $\mathfrak{G}_{\omega}$ such that $p_1 | p_2$.  Then any monomial $q(x_1, \ldots, x_n)$ satisfying $p_1 | q $ and $q | p_2$ must also have nonzero coefficient in $\mathfrak{G}_{\omega}$.
\end{theorem}
\begin{proof}
Suppose that $(P_1,S_1)$ and $(P_2,S_2)$ are marked bumpless pipe dreams in $MPipes(\omega)$ corresponding, respectively, to the monomials $p_1$ and $p_2$.  Note that if $p_1$ and $p_2$ have equal degree, then $p_1=p_2$, and the result becomes trivial.  We can, thus, assume without loss of generality that $p_1$ has strictly lower degree than $p_2$, and in particular, $p_2$ does not have minimal degree in $\mathfrak{G}_{\omega}$.  Since vexillary permutations have only reduced BPDs, this implies that $S_2$ must be non-empty.  \\ \\
Consider any $i$ such that $x_i$ has a lower exponent in $p_1$ than in $p_2$.  If $S_2$ contains a square $(i,j)$ in row $i$, then the marked BPD $(P_2,S_2 - \{(i,j)\}) \in MPipes (\omega)$ corresponds to the monomial $p_2 / x_i$.  Otherwise, there must be some square in row $i$ which is blank in $P_2$, but not in $P_1$.  It must, thus, be possible to perform one of the following types of moves on the BPD $P_2$ (where shaded squares represent squares in the set $S_2$). \\ \\
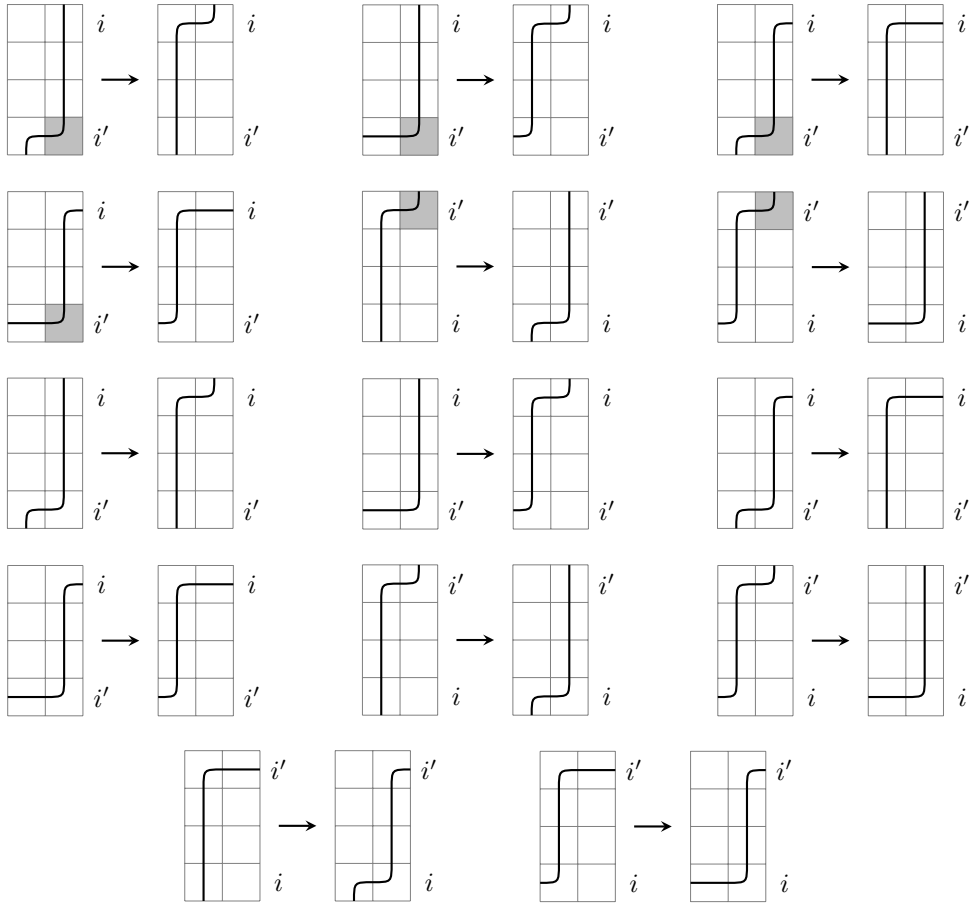
\begin{center}
    \scalebox{1}{\begin{tikzpicture}
    \filldraw[draw=gray, color=lightgray] (.5,0) rectangle (1,.5);
\draw[step=.5cm,gray,very thin] (0,0) grid (1,2);
\node at (1.25,1.75) {$i$};
\node at (1.25,.25) {$i'$};
\draw[black, thick] (.25,0) .. controls (.25,.25) .. (.5,.25) .. controls (.75,.25) .. (.75,.5) -- (.75,2);
\draw[-stealth, black, thick] (1.25,1) -- (1.75,1);
\draw[step=.5cm,gray,very thin] (1.999,0) grid (3,2);
\node at (3.25,1.75) {$i$};
\node at (3.25,.25) {$i'$};
\draw[black,thick] (2.25,0) -- (2.25,1.5) .. controls (2.25,1.75) .. (2.5,1.75) .. controls (2.75,1.75) .. (2.75,2);
\end{tikzpicture}} \hspace{1 cm}            
\scalebox{1}{\begin{tikzpicture}
\filldraw[draw=gray, color=lightgray] (.5,0) rectangle (1,.5);
\draw[step=.5cm,gray,very thin] (0,0) grid (1,2);
\node at (1.25,1.75) {$i$};
\node at (1.25,.25) {$i'$};
\draw[black, thick] (0,.25) -- (.5,.25) .. controls (.75,.25) .. (.75,.5) -- (.75,2);
\draw[-stealth, black, thick] (1.25,1) -- (1.75,1);
\draw[step=.5cm,gray,very thin] (1.999,0) grid (3,2);
\node at (3.25,1.75) {$i$};
\node at (3.25,.25) {$i'$};
\draw[black,thick] (2,.25).. controls (2.25,.25) .. (2.25,.5) -- (2.25,1.5) .. controls (2.25,1.75) .. (2.5,1.75) .. controls (2.75,1.75) .. (2.75,2);
\end{tikzpicture}} \hspace{1 cm}                
\scalebox{1}{\begin{tikzpicture}
\filldraw[draw=gray, color=lightgray] (.5,0) rectangle (1,.5);
\draw[step=.5cm,gray,very thin] (0,0) grid (1,2);
\node at (1.25,1.75) {$i$};
\node at (1.25,.25) {$i'$};
\draw[black, thick] (.25,0) .. controls (.25,.25) .. (.5,.25) .. controls (.75,.25) .. (.75,.5) -- (.75,1.5) .. controls (.75,1.75) .. (1,1.75);
\draw[-stealth, black, thick] (1.25,1) -- (1.75,1);
\draw[step=.5cm,gray,very thin] (1.999,0) grid (3,2);
\node at (3.25,1.75) {$i$};
\node at (3.25,.25) {$i'$};
\draw[black,thick] (2.25,0) -- (2.25,1.5) .. controls (2.25,1.75) .. (2.5,1.75) -- (3,1.75);
\end{tikzpicture}} 

\scalebox{1}{\begin{tikzpicture}
\filldraw[draw=gray, color=lightgray] (.5,0) rectangle (1,.5);
\draw[step=.5cm,gray,very thin] (0,0) grid (1,2);
\node at (1.25,1.75) {$i$};
\node at (1.25,.25) {$i'$};
\draw[black, thick] (0,.25) -- (.5,.25) .. controls (.75,.25) .. (.75,.5) -- (.75,1.5) .. controls (.75,1.75) .. (1,1.75);
\draw[-stealth, black, thick] (1.25,1) -- (1.75,1);
\draw[step=.5cm,gray,very thin] (1.999,0) grid (3,2);
\node at (3.25,1.75) {$i$};
\node at (3.25,.25) {$i'$};
\draw[black,thick] (2,.25).. controls (2.25,.25) .. (2.25,.5) -- (2.25,1.5) .. controls (2.25,1.75) .. (2.5,1.75) -- (3,1.75);
\end{tikzpicture}} \hspace{1 cm}  
\scalebox{1}{\begin{tikzpicture}
\filldraw[draw=gray, color=lightgray] (.5,1.5) rectangle (1,2);
\draw[step=.5cm,gray,very thin] (0,0) grid (1,2);
\node at (1.25,1.75) {$i'$};
\node at (1.25,.25) {$i$};
\draw[black,thick] (.25,0) -- (.25,1.5) .. controls (.25,1.75) .. (.5,1.75) .. controls (.75,1.75) .. (.75,2);
\draw[-stealth, black, thick] (1.25,1) -- (1.75,1);
\draw[step=.5cm,gray,very thin] (1.999,0) grid (3,2);
\node at (3.25,1.75) {$i'$};
\node at (3.25,.25) {$i$};
\draw[black,thick] (2.25,0) .. controls (2.25,.25) .. (2.5,.25) .. controls (2.75,.25) .. (2.75,.5) -- (2.75,2);
\end{tikzpicture}} \hspace{1 cm}  
\scalebox{1}{\begin{tikzpicture}
\filldraw[draw=gray, color=lightgray] (.5,1.5) rectangle (1,2);
\draw[step=.5cm,gray,very thin] (0,0) grid (1,2);
\node at (1.25,1.75) {$i'$};
\node at (1.25,.25) {$i$};
\draw[black,thick] (0,.25) .. controls (.25,.25) .. (.25,.5) -- (.25,1.5) .. controls (.25,1.75) .. (.5,1.75) .. controls (.75,1.75) .. (.75,2);
\draw[-stealth, black, thick] (1.25,1) -- (1.75,1);
\draw[step=.5cm,gray,very thin] (1.999,0) grid (3,2);
\node at (3.25,1.75) {$i'$};
\node at (3.25,.25) {$i$};
\draw[black,thick] (2,.25) -- (2.5,.25) .. controls (2.75,.25) .. (2.75,.5) -- (2.75,2);
\end{tikzpicture}} 

    \scalebox{1}{\begin{tikzpicture}
\draw[step=.5cm,gray,very thin] (0,0) grid (1,2);
\node at (1.25,1.75) {$i$};
\node at (1.25,.25) {$i'$};
\draw[black, thick] (.25,0) .. controls (.25,.25) .. (.5,.25) .. controls (.75,.25) .. (.75,.5) -- (.75,2);
\draw[-stealth, black, thick] (1.25,1) -- (1.75,1);
\draw[step=.5cm,gray,very thin] (1.999,0) grid (3,2);
\node at (3.25,1.75) {$i$};
\node at (3.25,.25) {$i'$};
\draw[black,thick] (2.25,0) -- (2.25,1.5) .. controls (2.25,1.75) .. (2.5,1.75) .. controls (2.75,1.75) .. (2.75,2);
\end{tikzpicture}} \hspace{1 cm}   
\scalebox{1}{\begin{tikzpicture}
\draw[step=.5cm,gray,very thin] (0,0) grid (1,2);
\node at (1.25,1.75) {$i$};
\node at (1.25,.25) {$i'$};
\draw[black, thick] (0,.25) -- (.5,.25) .. controls (.75,.25) .. (.75,.5) -- (.75,2);
\draw[-stealth, black, thick] (1.25,1) -- (1.75,1);
\draw[step=.5cm,gray,very thin] (1.999,0) grid (3,2);
\node at (3.25,1.75) {$i$};
\node at (3.25,.25) {$i'$};
\draw[black,thick] (2,.25).. controls (2.25,.25) .. (2.25,.5) -- (2.25,1.5) .. controls (2.25,1.75) .. (2.5,1.75) .. controls (2.75,1.75) .. (2.75,2);
\end{tikzpicture}} \hspace{1 cm}  
\scalebox{1}{\begin{tikzpicture}
\draw[step=.5cm,gray,very thin] (0,0) grid (1,2);
\node at (1.25,1.75) {$i$};
\node at (1.25,.25) {$i'$};
\draw[black, thick] (.25,0) .. controls (.25,.25) .. (.5,.25) .. controls (.75,.25) .. (.75,.5) -- (.75,1.5) .. controls (.75,1.75) .. (1,1.75);
\draw[-stealth, black, thick] (1.25,1) -- (1.75,1);
\draw[step=.5cm,gray,very thin] (1.999,0) grid (3,2);
\node at (3.25,1.75) {$i$};
\node at (3.25,.25) {$i'$};
\draw[black,thick] (2.25,0) -- (2.25,1.5) .. controls (2.25,1.75) .. (2.5,1.75) -- (3,1.75);
\end{tikzpicture}} 

\scalebox{1}{\begin{tikzpicture}
\draw[step=.5cm,gray,very thin] (0,0) grid (1,2);
\node at (1.25,1.75) {$i$};
\node at (1.25,.25) {$i'$};
\draw[black, thick] (0,.25) -- (.5,.25) .. controls (.75,.25) .. (.75,.5) -- (.75,1.5) .. controls (.75,1.75) .. (1,1.75);
\draw[-stealth, black, thick] (1.25,1) -- (1.75,1);
\draw[step=.5cm,gray,very thin] (1.999,0) grid (3,2);
\node at (3.25,1.75) {$i$};
\node at (3.25,.25) {$i'$};
\draw[black,thick] (2,.25).. controls (2.25,.25) .. (2.25,.5) -- (2.25,1.5) .. controls (2.25,1.75) .. (2.5,1.75) -- (3,1.75);
\end{tikzpicture}} \hspace{1 cm}  
\scalebox{1}{\begin{tikzpicture}
\draw[step=.5cm,gray,very thin] (0,0) grid (1,2);
\node at (1.25,1.75) {$i'$};
\node at (1.25,.25) {$i$};
\draw[black,thick] (.25,0) -- (.25,1.5) .. controls (.25,1.75) .. (.5,1.75) .. controls (.75,1.75) .. (.75,2);
\draw[-stealth, black, thick] (1.25,1) -- (1.75,1);
\draw[step=.5cm,gray,very thin] (1.999,0) grid (3,2);
\node at (3.25,1.75) {$i'$};
\node at (3.25,.25) {$i$};
\draw[black,thick] (2.25,0) .. controls (2.25,.25) .. (2.5,.25) .. controls (2.75,.25) .. (2.75,.5) -- (2.75,2);
\end{tikzpicture}} \hspace{1 cm}  
\scalebox{1}{\begin{tikzpicture}
\draw[step=.5cm,gray,very thin] (0,0) grid (1,2);
\node at (1.25,1.75) {$i'$};
\node at (1.25,.25) {$i$};
\draw[black,thick] (0,.25) .. controls (.25,.25) .. (.25,.5) -- (.25,1.5) .. controls (.25,1.75) .. (.5,1.75) .. controls (.75,1.75) .. (.75,2);
\draw[-stealth, black, thick] (1.25,1) -- (1.75,1);
\draw[step=.5cm,gray,very thin] (1.999,0) grid (3,2);
\node at (3.25,1.75) {$i'$};
\node at (3.25,.25) {$i$};
\draw[black,thick] (2,.25) -- (2.5,.25) .. controls (2.75,.25) .. (2.75,.5) -- (2.75,2);
\end{tikzpicture}} 

\scalebox{1}{\begin{tikzpicture}
\draw[step=.5cm,gray,very thin] (0,0) grid (1,2);
\node at (1.25,1.75) {$i'$};
\node at (1.25,.25) {$i$};
\draw[black,thick] (.25,0) -- (.25,1.5) .. controls (.25,1.75) .. (.5,1.75) -- (1,1.75);
\draw[-stealth, black, thick] (1.25,1) -- (1.75,1);
\draw[step=.5cm,gray,very thin] (1.999,0) grid (3,2);
\node at (3.25,1.75) {$i'$};
\node at (3.25,.25) {$i$};
\draw[black,thick] (2.25,0) .. controls (2.25,.25) .. (2.5,.25) .. controls (2.75,.25) .. (2.75,.5) -- (2.75,1.5) .. controls (2.75,1.75) .. (3,1.75);
\end{tikzpicture}} \hspace{1 cm}  
\scalebox{1}{\begin{tikzpicture}
\draw[step=.5cm,gray,very thin] (0,0) grid (1,2);
\node at (1.25,1.75) {$i'$};
\node at (1.25,.25) {$i$};
\draw[black,thick] (0,.25) .. controls (.25,.25) .. (.25,.5) -- (.25,1.5) .. controls (.25,1.75) .. (.5,1.75) -- (1,1.75);
\draw[-stealth, black, thick] (1.25,1) -- (1.75,1);
\draw[step=.5cm,gray,very thin] (1.999,0) grid (3,2);
\node at (3.25,1.75) {$i'$};
\node at (3.25,.25) {$i$};
\draw[black,thick] (2,.25) -- (2.5,.25) .. controls (2.75,.25) .. (2.75,.5) -- (2.75,1.5) .. controls (2.75,1.75) .. (3,1.75);
\end{tikzpicture}}
\end{center}
Select a move of one of these forms such that the resulting pipes will be closer to their positions in $P_1$.  Note that the first six of these possibilities result in a new marked BPD corresponding to the monomial $p_2/x_i$.  For the remaining options, the new marked BPD corresponds to $p_2 \frac{x_{i'}}{x_i}$.  In the latter case, the exponent of $x_{i'}$ is now greater than in $p_1$, so we can either remove a row $i'$ square from $S_2$ to give a marked BPD corresponding to $p_2/x_i$, or we can again perform one of the above moves to remove a blank square from row $i'$. \\ \\  
Since every step moves the pipes closer to their positions in $P_1$ and $S_2$ is non-empty, this process will eventually end with a marked BPD corresponding to $p_2/x_i$.  The result then follows by induction. 
\end{proof} 
\noindent We now address Weigandt's question regarding the leading terms in each degree of the Grothendieck polynomial \cite{WeigandtTalk}.  Let $\mathfrak{G}_{\omega}^{(k)}(\mathbf{x})$ be the homogeneous component of $\mathfrak{G}_{\omega}(\mathbf{x})$ with degree k, and consider a term order with $x_1<x_2< \cdots <x_n$.  The following results are known.  \\
\begin{proposition}[\cite{RCgraphs}, Corollary 3.9]
\label{SLeading}
The leading term of the Schubert polynomial $\mathfrak{S}_{\omega}(\mathbf{x})$ is given by $x_1^{c_1} \cdots x_n^{c_n}$.
\end{proposition} 
\begin{theorem}[\cite{Speyerslides}, Theorem 1.1]
\label{GLeading}
The leading term of the Grothendieck polynomial $\mathfrak{G}_{\omega}(\mathbf{x})$ is a scalar multiple of $x_1^{r_1} \cdots x_n^{r_n}$.
\end{theorem}
\noindent In particular, we know the leading terms of the lowest and highest degree homogeneous components of $\mathfrak{G}_{\omega}(\mathbf{x})$.  Furthermore, for the component of second lowest degree, the leading term is given as follows.
\begin{proposition}
\label{SPlusOneLeading}
For any permutation $\omega$, the leading term of $\mathfrak{G}_{\omega}^{(\ell(\omega)+1)}(\mathbf{x})$ is given by a scalar multiple of $(x_1^{c_1} \cdots x_n^{c_n}) x_i$ where $i$ is the largest possible index such that $c_i < r_i$.
\end{proposition}
\begin{proof}
Let $(P,S)$ be a marked BPD corresponding to the leading term of $\mathfrak{G}_{\omega}^{(\ell(\omega)+1)}(\mathbf{x})$, and recall that $P$ can be reached from $P_{R}(\omega)$ for $\omega$ by a series of droops and K-theoretic droops.  In particular, since the corresponding monomial has degree $\ell(\omega)+1$, it must be true that either this sequence contains no K-theoretic droop moves and $|S|=1$ or this sequence contains exactly one K-theoretic droop and $|S|=0$.  \\ \\
Consider first the case where the sequence contains a K-theoretic droop.  Note that performing a droop move on a BPD will always generate a monomial of the same degree which is smaller in our term order, so the K-theoretic droop must be the last move of our sequence.  Now, suppose that instead of performing one of the following K-theoretic droops, \\
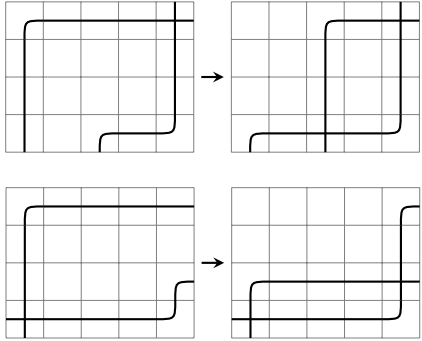
\begin{center}
   \scalebox{1}{\begin{tikzpicture}
\draw[step=.5cm,gray,very thin] (0,0) grid (2.5,2);
\draw[step=.5cm,gray,very thin] (2.999,0) grid (5.5,2);
\draw[-stealth, black, thick] (2.6,1) -- (2.9,1);
\draw[black, thick] (.25,0) -- (.25,1.5);
\draw[black, thick] (.25,1.5) .. controls (.25,1.75) .. (.5,1.75);
\draw[black,thick] (.5,1.75) -- (2.5,1.75);
\draw[black,thick] (1.25,0) .. controls (1.25,.25) .. (1.5,.25) -- (2,.25) .. controls (2.25,.25) .. (2.25,.5) -- (2.25,2);

\draw[black,thick] (3.25,0) .. controls (3.25,.25) .. (3.5,.25);
\draw[black,thick] (3.5,.25) -- (5,.25);
\draw[black,thick] (5, .25) .. controls (5.25,.25) .. (5.25,.5);
\draw[black,thick] (5.25,.5)--(5.25,2);
\draw[black,thick] (4.25,0) -- (4.25,1.5) .. controls (4.25,1.75) .. (4.5,1.75) -- (5.5,1.75);
\end{tikzpicture}}

    \scalebox{1}{\begin{tikzpicture}
\draw[step=.5cm,gray,very thin] (0,0) grid (2.5,2);
\draw[step=.5cm,gray,very thin] (2.999,0) grid (5.5,2);
\draw[-stealth, black, thick] (2.6,1) -- (2.9,1);
\draw[black, thick] (.25,0) -- (.25,1.5);
\draw[black, thick] (.25,1.5) .. controls (.25,1.75) .. (.5,1.75);
\draw[black,thick] (.5,1.75) -- (2.5,1.75);
\draw[black,thick] (0,.25) -- (2,.25) .. controls (2.25,.25) .. (2.25,.5) .. controls (2.25,.75) .. (2.5,.75);

\draw[black,thick] (3,.25) -- (5,.25);
\draw[black,thick] (5, .25) .. controls (5.25,.25) .. (5.25,.5);
\draw[black,thick] (5.25,.5)--(5.25,1.5);
\draw[black,thick] (5.25,1.5).. controls (5.25,1.75) .. (5.5,1.75);
\draw[black,thick] (3.25,0) -- (3.25,.5) .. controls (3.25,.75) .. (3.5,.75) -- (5.5,.75);
\end{tikzpicture}}
\end{center} 
we add the up-elbow in the lower right corner to $S$.  This new BPD would correspond to a larger monomial in our term order, contradicting the assumption that $(P,S)$ corresponded to the leading term.  \\ \\
We can, thus, conclude that $P$ must be reduced and $S$ must have size $1$.  By the same argument as in the proof of Theorem \ref{leadingterms}, we can see that $P$ will have no up-elbows besides the one which is in $S$.  To get such a BPD corresponding to the largest possible monomial, we take the lowest pipe in $P_R$ which can be drooped and droop it to the closest open square.  This lowest pipe will be pipe $\omega(i)$ where $i$ is the largest index with $c_i < r_i$, so $(P,S)$ will correspond to $(x_1^{c_1} \cdots x_n^{c_n}) x_i$.
\end{proof} 
\noindent We conjecture the following description for the leading term of each intermediate degree and give a proof of this description in the vexillary case.
\begin{conjecture}
\label{conjleading}
For any permutation $\omega$, fix $k$ such that $deg(\mathfrak{S}_{\omega}(x)) \leq k \leq deg(\mathfrak{G}_{\omega}(x))$.  Let $j \in \{1,\ldots,n\}$ be the smallest possible index such that $\sum_{i=j}^n(r_i-c_i) \leq k-deg(\mathfrak{S}_{\omega}(x))$.  If we set $e_i=r_i$ for all $i \geq j$, $e_{j-1}=c_{j-1}+k-deg(\mathfrak{S}_{\omega}(x))-\sum_{i=j}^n(r_i-c_i)$, and $e_i=c_i$ for all $i<j-1$, then the leading term of $\mathfrak{G}_{\omega}^{(k)}(\mathbf{x})$ is a scalar multiple of $x_1^{e_1} \cdots x_n^{e_n}$.
\end{conjecture}
\begin{theorem2}
Suppose that $\omega$ is vexillary, and fix $k$ such that $deg(\mathfrak{S}_{\omega}(x)) \leq k \leq deg(\mathfrak{G}_{\omega}(x))$.  Let $j \in \{1,\ldots,n\}$ be the smallest possible index such that $\sum_{i=j}^n(r_i-c_i) \leq k-deg(\mathfrak{S}_{\omega}(x))$.  If we set $e_i=r_i$ for all $i \geq j$, $e_{j-1}=c_{j-1}+k-deg(\mathfrak{S}_{\omega}(x))-\sum_{i=j}^n(r_i-c_i)$, and $e_i=c_i$ for all $i<j-1$, then the leading term of $\mathfrak{G}_{\omega}^{(k)}(\mathbf{x})$ is a scalar multiple of $x_1^{e_1} \cdots x_n^{e_n}$.
\end{theorem2}
\begin{proof}
Let $(P,S)$ be a marked BPD corresponding to the leading term of $\mathfrak{G}_{\omega}^{(k)}(\mathbf{x})$, and note that $S$ will have size $k-\ell(\omega)$ since $\ell(\omega)=\sum_i c_i$ is the degree of $\mathfrak{S}_{\omega}(\mathbf{x})$.  First, suppose that $S$ is a proper subset of $U(P)$ and choose $(i,j) \in U(P)-S$ such that $i$ is as small as possible.  If $S$ contains some $(i',j')$ such that $i' < i$, then $(P,(S-(i',j')) \cup \{(i,j)\})$ corresponds to a larger monomial in $\mathfrak{G}_{\omega}^{(k)}$ giving a contradiction.  Similarly, if $P$ contains no up-elbows above row $i$, then un-drooping the elbow in $(i,j)$ gives a marked BPD corresponding to a larger monomial, providing our contradiction.  We, therefore, know that the leading monomial must correspond to a marked BPD of the form $(P,U(P))$.  \\ \\
We now know that $P$ must have exactly $k-\ell(\omega)$ up-elbows and that pipe $\omega(i)$ can have at most $r_i-c_i$ up-elbows, so we can see that for $(P,U(P))$ to correspond to the leading term, $P$ must be constructed as follows.  For the largest $i$ such that $r_i > c_i$, we droop pipe $\omega(i)$ so that it has $min\{r_i-c_i,k-\ell(\omega)\}$ up-elbows pushed as far left as possible.  Now, set $r$ to be the difference between $k-\ell(\omega)$ and the total number of up-elbows in our BPD.  If $r=0$, we are done.  Otherwise, consider the next largest $i'$ such that $r_{i'} > c_{i'}$, and add $min\{r_{i'}-c_{i'},r\}$ up-elbows to pipe $\omega(i')$ (again, pushed as far left as possible), repeating this process until $r=0$.  By construction, the resulting BPD will correspond to $x_1^{e_1} \cdots x_n^{e_n}$, as described above, so we can conclude that this is the leading monomial of $\mathfrak{G}_{\omega}^{(k)}(\mathbf{x})$ (up to multiplication by a scalar).
\end{proof}
\section*{Acknowledgements}
The author would like to thank Karola M\'esz\'aros, Avery St.~Dizier, Allen Knutson, Ed Swartz, and Colleen Robichaux for helpful discussions and feedback on earlier drafts of this work.  
\bibliographystyle{amsplain}
\bibliography{Vexillary_BPDs}
\end{document}